\documentclass[letterpaper, 9 pt,conference,twocolumn]{article}

\usepackage[english]{babel} 

 \usepackage{amsmath,amsfonts,amssymb,amsthm}
 \usepackage{graphicx} 
 \usepackage{subcaption}
 \usepackage{hyperref}
 \usepackage{enumitem}
 \usepackage{extarrows}
 \usepackage[export]{adjustbox}
 \usepackage{accents}
 \newcommand{\dbtilde}[1]{\accentset{\approx}{#1}}
\numberwithin{equation}{section}

\newtheorem{theorem}{Theorem}[section]
\newtheorem{corollary}[theorem]{Corollary}
\newtheorem{lemma}[theorem]{Lemma}
\newtheorem{proposition}[theorem]{Proposition}

\newtheorem{remark}[theorem]{Remark} 

\renewcommand{\leq}{\leqslant}

\renewcommand{\geq}{\geqslant}

\def\N{\mathbb{N}}

\def\R{\mathbb{R}}

\usepackage{hyperref} 

\hypersetup{
	hidelinks,
	colorlinks,
	breaklinks=true,
	urlcolor=black,
	citecolor=black,
	linkcolor=black,
	bookmarksopen=false,
	pdftitle={Title},
	pdfauthor={Author},
}

\title{Interpolation and approximation via Momentum ResNets and Neural ODEs}
\author{ Domènec Ruiz-Balet$^{1,2}$, Elisa Affili$^1$, Enrique Zuazua$^{1,2,3}$
\thanks{$^1$ Chair of Computational Mathematics, Fundación Deusto Av. de las Universidades 24, 48007 Bilbao, Basque Country, Spain}
\thanks{$^2$ Departamento de Matemáticas, Universidad Autónoma de Madrid, 28049 Madrid, Spain}
\thanks{$^3$ Chair in Applied Analysis, Alexander von Humboldt-Professorship, Department of Data Science Friedrich-Alexander-Universität, Erlangen-Nürnberg, 91058 Erlangen, Germany}}

\begin{document}

\maketitle
\thispagestyle{empty}
\pagestyle{empty}

\begin{abstract}
{\small
In this article, we explore the effects of memory terms in continuous-layer Deep Residual Networks by studying Neural ODEs (NODEs). We investigate two types of models.
On one side, we consider the case of Residual Neural Networks with dependence on multiple layers,  more precisely  Momentum ResNets.
On the other side, we analyse a Neural ODE with auxiliary states playing the role of memory states.

We examine the interpolation and universal approximation properties for both architectures through a simultaneous control perspective. 
We also prove the ability of the second model to represent sophisticated maps, such as parametrizations of time-dependent functions. 
Numerical simulations complement our study.

}
\end{abstract}



\section{Introduction}
\subsection{Objectives}
In the last decades, ever-increasing computational power has enabled the development of Neural Networks with complex architectures to perform in different tasks in Machine Learning.
However, the power of these tools is often measured by engineering experiments, with no deep analysis of the model's potential and limits. 
In this paper, we analyse interpolation and approximation properties of two models involving memory terms in some fundamental tasks of Supervised Learning, namely the representation of functions. We will show in particular that these models have advantages with respect to the Neural ODEs previously studied in \cite{ruizbalet2021neural}.

We first study the continuous-layer version of 
Momentum ResNets, recently proposed in \cite{sander2021momentum}, with a ReLU activation function.  Precisely, the Neural ODE we investigate is
\begin{equation}\label{eq:resnet}
\ddot{x}+\dot{x}=w(t)\sigma(\langle a(t),x\rangle+b(t)),    
\end{equation}
where $x(t)\in \mathbb{R}^d$, $\sigma:\mathbb{R}\to \mathbb{R}$ is the ReLU activation function, $a,w\in L^\infty((0,T),\mathbb{R}^d)$ and $b\in L^\infty((0,T),\mathbb{R})$.  We will prove an interpolation result 
and the universal approximation property.

Then, we introduce a class of Neural ODEs with auxiliary states, $p$, representing the memory of the system, namely
\begin{equation}\label{extended}
\begin{cases}
    \dot{x}=W\boldsymbol{\vec{\sigma}}(Ax  + Cp+ b_1) +b_2,\\
    \dot{p}=u\sigma(\langle d,x \rangle +f),
\end{cases}
\end{equation}
where $x(t)\in \mathbb{R}^d$ and $p(t)\in \mathbb{R}^{d_p}$, $C\in L^\infty((0,T); \mathbb{R}^{d\times d_p})$, $W,A\in L^\infty((0,T); \mathbb{R}^{d\times d})$, $d,b_1,b_2\in L^\infty((0,T); \mathbb{R}^{d}) $, $f\in L^\infty((0,T); \mathbb{R}) $ and $\boldsymbol{\vec{\sigma}}:\mathbb{R}^{d}\to \mathbb{R}^d$ is the componentwise application of the activation function.
We will prove an interpolation result for maps giving a continuous parameterization of time-dependent functions, precisely
\begin{equation}\label{function} M:\Omega\subset\mathbb{R}^d\to C([0,T];\mathbb{R}^d)\cap BV([0,T];\mathbb{R}^d),
\end{equation}
with $\Omega$ a bounded set.
We first investigate the ability of this type of NODEs  to interpolate samples of $M$ for all times in $(0,T)$, and then to approximate $M$ through the flow of \eqref{extended}. 
These correspond, respectively, to the properties  of simultaneous tracking controllability and universal tracking approximation.

\subsection{Context and notation}

{\color{black}
\textit{Supervised Learning} is the branch of Machine Learning that studies how to learn a function based on a finite collection of samples. 
To obtain a solution for these problems, one typically uses a parameterized class of functions; then, the parameters that correspond to the best candidate in the class are found through optimization techniques. 

}

\textit{Neural Networks} (NN) \cite{lecun2015deep} are a common choice of such parameterized classes of functions. In this article, we focus on continuous layer versions of \textit{Residual Neural Networks} (ResNets), introduced in \cite{he2016deep}, which have been observed to outperform of previous classes \cite{kolesnikov2020big}, especially in tasks where a time structure is present, like in text and video recognition. ResNets are a collection of $N_{l}\in\N$ consecutive layers of neurons, where the initial data $\{x_{i}\}_{i=1}^N$, $N\in\N$, are transformed following a relation of the type
{\small
\begin{equation}\label{resnetI}
\hspace{-0.3cm}
\begin{cases}
    x_{k+1,i}= x_k + g(x_k;\omega_k),\ \  k\in\{0,...,N_{\mathrm{l}}-1\},\\
    x_{0,i}=x_i\in\mathbb{R}^d.
\end{cases}
\end{equation}
}
Here, $\omega_k$ are the parameters to be chosen. The peculiarity of ResNets with respect to previous deep NN architectures is the presence of an inertia term $x_k$ in the definition of the term $x_{k+1}$.
To fix the ideas, for the Momentum ResNets,  $g: \R^{3d+1} \to \R^d$ it is going to be
\begin{equation}\label{g}
   g(x;w,a,b):=w\sigma (\langle a,x \rangle +b),
\end{equation}
where $\omega:=\left\{w_k,a_k,b_k\right\}_{k=1}^{N_{\mathrm{l}}}$ are parameters such that $w_k,a_k\in \mathbb{R}^d$ and $b_k\in\mathbb{R}$, and $\langle \cdot, \cdot \rangle$ denotes the scalar product. 
{\color{black}The function $\sigma$ is the so-called \textit{activation function}. When not stated otherwise,  $\sigma:\R\to \R$ will denote the \textit{Rectified Linear Unit} (ReLU)  defined by
\begin{equation}\label{ReLU}
    \sigma (x): = \max \{ x, 0\}.
\end{equation}}
In practical situations, the controls are found through optimization techniques. In the context of continuous-layer models, this is equivalent to solve optimal control problems, see \cite{barcena2021optimal,esteve2020large,yague2021sparse}.

The ResNet architecture \eqref{resnetI} can be seen as an Euler discretization of a continuous time model, see \cite{chen2018neural,esteve2020large,haber2017stable,weinan2017proposal}. 
In fact, as the number of layers $N_l$ goes to infinity, one can understand \eqref{resnetI} as the discretization of the ODE  
\begin{equation}\label{NODE}
    \dot{x}=g(x; \omega(t)),
\end{equation}
as was done in \cite{weinan2017proposal}.
By reading the parameters $\omega(t)$ as time-dependant functions, we can solve the approximation problems from an ODE point of view. 
In fact, for all times $T\geq 0$ there exists a flow
\begin{equation}\label{flux}
    \phi_T(x; w,a,b),
\end{equation}
which is the solution to the equation \eqref{NODE} at $T$ for the initial condition $x$ and functions $w$, $a$ and $b$. We will read $w$, $a$ and $b$ as the controls of the equation \eqref{NODE}.
We will ask if the flow \eqref{flux} at a certain time $T$ is able to approximate a desired function for suitable controls.

One basic question to address is if a NN is able to \emph{interpolate} any set of samples of the target function; that is, given a sample of initial points $\{x_i\}_{i=1}^N\subset \R^d $, with $d$ and $N$ in $\N$, and of target points $\{y_i\}_{i=1}^N\subset \R^d$, we ask if there exist a function $\omega$ such that for all $i=1, \dots, N$ we have that
$\phi_T(x_i, \omega)=y_i$. This is a \emph{simultaneous control} problem \cite{lions1988exact,loheac2016averaged}. Along the paper,  interpolation is equivalent to a simultaneous control problem. 
This issue has been considered, for instance, in  \cite{ruizbalet2021neural} through explicit constructions and in \cite{cuchiero2020deep} using Lie Bracket techniques. 

Another important theoretical question to ask is if, given a function $f: \Omega \subset \R^d\to\R^d$, $d\in\N$, $\Omega$ a bounded set, a NN is able to approximate $f$ as accurately as desired on $\Omega$, for example in the $L^2$ norm. This property of the NN is called \textit{universal approximation}, see for instance \cite{cybenko1989approximation,daubechies2021nonlinear,guhring2021approximation,hornik1989multilayer,li2019deep,pinkus1999approximation}.

Now we turn our attention to being able to approximate more general maps, as the one in \eqref{function}, through a Neural ODE. 
We consider a collection of samples of $M$, $\{y_i=M(x_i)\}_{i=1}^N$ with $y_i\in C^0((0,T);\mathbb{R}^d)\cap BV((0,T);\mathbb{R}^d)$. 
We say that the dynamics \eqref{extended} is \textit{simultaneously tracking-controllable} if there exist a control functions $\omega$ such that the associated flow $\phi_t(x_i,\omega)$ is close to $y_i(t)$ for every $i$ and $t\in[0,T]$. 
The tracking controllability has been studied in some problems in control of PDEs, for instance see \cite{barcena2021optimal}.

We will also discuss the \textit{universal tracking approximation}, where the goal is to approximate the map $M$  with the flow given by the NODE \eqref{extended}.

In \cite{ruizbalet2021neural}, the authors considered the NODE
\begin{equation}\label{NODE}
\dot{x}=w(t)\sigma(\langle a(t),x\rangle+b),
\end{equation}
that from now on, to make a clear distinction, we will call a \emph{first-oder NODE}.
When the initial data and the targets are in $\mathbb{R}^d$, the dynamics of \eqref{NODE}
is not able to interpolate exactly any set of points.
This happens because, by uniqueness of the solution of \eqref{NODE}, one cannot drive two distinct points to the same target; so, if two targets coincide, there cannot be simultaneous exact controllability.

\subsection{The Momentum ResNets}
In the recent paper \cite{sander2021momentum}, the authors proposed a new model of continuous layer ResNets. 
Precisely, instead of \eqref{NODE}, they use the equation with parameters $\theta \in \R^{d\times d}$ given by
\begin{equation*}\label{momnet}
    \varepsilon \ddot{x} + \dot{x} = g(x, \theta), \quad \text{with} \ (x(0),\dot{x}(0))=(x_i, p_i),
\end{equation*}
which is called a \textit{Momentum ResNet}. 
 A major difference 
with respect to NODEs is that two points can share the same position and be distinguished by their velocity. Moreover, Momentum ResNets allows the authors of \cite{sander2021momentum} to use a backpropagation algorithm requiring less memory with respect to the NODEs. 
In this framework, the parameter $\varepsilon$ acts as a measure of the memory of the system, seen as the dependence on previous layers. When $\varepsilon\to 0$, one recovers the NODE.
Regarding the fundamental properties of the Momentum ResNets, the authors investigate the class of functions that can be recovered via the Momentum ResNets with a linear function $g(x)=\theta\, x$, $\theta\in\R^{d\times d}$, proving that they can recover diagonalizable matrices satisfying a certain condition on their eigenvalues.

In this paper, the first model we investigate is the Momentum Resnet
\begin{equation}\label{momnetI}
\begin{cases}
    \ddot{x} + \dot{x} = w(t)\sigma (\langle a(t),x(t) \rangle +b(t)), \\ (x(0),\dot{x}(0))=(x_i, p_i),
\end{cases}
\end{equation}
with $\sigma$ as in \eqref{ReLU} and $w$, $a \in L^{\infty}((0, T); \R^d)$, and $b\in L^{\infty}((0, T); \R)$.
Notice that, for the sake of simplicity of the mathematical dissertation, we took $\varepsilon=1$, thus taking into account the case of a strong memory dependence. 
However, the dynamical properties we will employ are not dependent on this parameter. 
We are going to investigate simultaneous controllability and universal approximation in $L^2$ for the Momentum ResNets as in \eqref{momnetI}. 
Following the intuitions in \cite{ruizbalet2021neural}, our proofs are relying on constructive arguments arising from the analysis of the exact dynamics of \eqref{momnetI}. 

We stress that Momentum ResNets can also be seen as memory like models. Indeed, the problem in \eqref{momnetI}
is equivalent to
\begin{equation}\label{eq:intmem}
\begin{cases}
   \dot{x}=-x+\int_0^tw(s)\sigma(\langle a(s),x\rangle +b(s))ds,\\
   x(0)=x_i,
  \end{cases}
\end{equation}
where we see that the past influences the dynamics of the system. Notice that we will have dependences on at least two layers when considering the discretization of the second derivative of the ODE.

\subsection{A class of Neural ODEs}
A different way to include memory in the system is to have some  auxiliary states playing the role of memory terms in the architecture. We will consider memory models of the form
{\small
\begin{equation}\label{eq:memory}
\begin{cases}
    x_{k+1}=x_k+g_x(x_k,p_k;\omega_k), &  k=1, \dots, N_l-1,   \\
     p_{k+1}=p_k+g_p(x_k,p_k;\omega_k),   &  k=1, \dots, N_l-1, 
\end{cases}
\end{equation}}
where $x \in \R^d$ is the state variable and $p\in \R^{d_p}$, $d_p\in\N$, is considered the memory variable. 
As for the ResNets, we will focus on a continuous time version of \eqref{eq:memory}, precisely
\begin{equation}\label{rnnI}
\begin{cases}
    \dot{x}=w\sigma(\langle a,x \rangle + \langle c,p \rangle+ b), \\
    \dot{p}=u\sigma(\langle d,x \rangle +f),
\end{cases}
\end{equation}
with $a$, $b$, $w$ as before, $c, u, d\in L^{\infty}((0, T); \R^d)$, and $f\in L^{\infty}((0, T); \R)$.
Note that, if we set the initial memory to be $0$, the equation can be written as:
$$
   \dot{x}=w\sigma\left(\langle a,x \rangle +\left\langle c, \int_0^t u\sigma\left(\langle d,x\rangle +f\right)ds\right\rangle +b\right).\
$$

Models with memory states, even continuous ones, are not new in the literature, see \cite{zhang2016learning} and the references therein.  
By using the properties of the first-order NODEs \eqref{NODE} analysed in \cite{ruizbalet2021neural},
it will be easy to deduce  simultaneously controllability and the universal approximation property for \eqref{rnnI}. 

The simultaneous controllability for \eqref{rnnI} allows to control both the state and the memory variable. This property is, for us, crucial to prove the simultaneous tracking controllability and universal tracking approximation. However, in order to do so, we need to extend the system \eqref{rnnI} to a more general one, given in \eqref{extended}. The model \eqref{extended} (as well as \eqref{rnnI}) could be understood as a Neural ODE in a higher dimensional space, in which we only consider some components as the state. 
In particular, we will consider the components of $p$ as memory states.
We will refer to \eqref{extended} as a \emph{Memory NODE}, to make a even clearer distinction with respect to the {first-order NODE} in \eqref{NODE}.

We are able to prove simultaneous tracking controllability for Memory NODEs; in particular, we will be able to control the trajectories on an interval $[0, T]$ for any time horizon $T>0$.
In our proof, the price to pay to obtain the simultaneous tracking controllability property is to consider a minimum dimension on the memory. 
In fact, we need $p\in\R^{d_p}$ with $d_p =2d$, which means that we have to consider a system that has bigger size with respect to \eqref{momnetI} and \eqref{NODE}. 
Moreover, also the controls take values in a higher dimensional space. 
However, we do not know if this is a fundamental requirement or if it is just technical.
Given the results of simultaneous tracking control and universal approximation, the {universal tracking approximation} follows as a corollary.

\subsection{Organization of the paper}

The paper is organized as follows:

\begin{itemize}
    \item In Section \ref{ss:mainres}, we state in detail the results concerning Momentum ResNets, and we will discuss the dynamical properties that allows us to prove the simultaneous controllability and the universal approximation theorem. The proofs for these theorems are in Appendix \ref{s:proofmr}.
    \item In Section \ref{ss:rnn}, we state the main results concerning Memory NODEs, the simultaneous control of the state-memory pair, the tracking simultaneous controllability, and we will prove the universal tracking approximation. The complete proofs for these theorems are in Appendix \ref{appsimcontr}.
    \item In Section \ref{ss:simus}, we will complete the study providing numerical simulations for those models in the discussed tasks. Specifically, we applied the systems to the case of indicator and sinusoidal functions.
    \item We will give some final comments in Section \ref{s:conslusion}.
\end{itemize}

\section{The Momentum ResNets}\label{ss:mainres}

In the first part of this section, we state and comment the results for the Momentum ResNets. 
In Subsection \ref{s:toolbox}, we will present and comment the dynamics of \eqref{momnetI}. In Subsection \ref{ss:ideas}, we will give a sketch of the proofs.
The complete proofs of the results are presented in the Appendix \ref{s:proofmr}.

\medskip
\textbf{Simultaneous controllability-Interpolation.}
We now consider the problem of simultaneous controllability, which is, driving with the same controls each point $x_i$ in the collection $\{x_i\}_{i=1}^N$, with $N\in\N$, to its given target $y_i$.  

\begin{theorem}\label{sim:control}
    Fix $T>0$, let $d\geq 2$ and let \newline $\{(x_i, 0),y_i\}_{i=1}^N\subset \mathbb{R}^{d}\times\mathbb{R}^d\times\mathbb{R}^d$, with $N\in\N$, such that $x_i\neq x_j$  if $i\neq j$.
 Then, there exist control functions $w,\ a\in L^\infty\left((0,T),\mathbb{R}^d\right)$ and $b\in L^\infty\left((0,T),\mathbb{R}\right)$ such that the flow associated to $a$, $w$, $b$ according to \eqref{momnetI} satisfies
 \begin{equation}\label{c:control}
     y_i=P_x\phi_T((x_i, p_i);a,w,b) \quad \forall i\in\{1,...,N\}.
 \end{equation}
 where $P_x$ is the projection to the $x$ component of the system.

\end{theorem}

\begin{remark}[Complexity of the controls]
The proof is constructive and makes use of piecewise constant controls. The maximum number of required  switches is of the order of N.
\end{remark}

The proof of Theorem \ref{sim:control} is similar to the proof of the Simultaneous Control Theorem in \cite{ruizbalet2021neural}. The idea is to construct a flow driving the points to their targets through piecewise constant controls, for which the behaviour of the system can be written explicitly.

An advantage of Momentum ResNets with respect to the first order Neural ODE is the exact interpolation for \emph{any} target. 

\medskip
Here we will show that one can make use of the flow of \eqref{momnetI} to have a universal approximation.

In \cite{ruizbalet2021neural}, one shows that the simultaneous control result plus the possibility of generating contractive flows in every Cartesian direction enables to have a universal approximation. 
In the case of Momentum ResNets, define
$$S_i:=\mathbb{R}^{i-1}\times\mathbb{R}_+\times \mathbb{R}^{d-i}$$
and let $V_i\subset \R^d$.
We say that the dynamics of \eqref{momnetI} can generate contractive flows in the $i-$th component if there exists some controls $\omega_i$ and an open set of velocities $V_i$, with $\{0\}^d\in V_i$, such that 
the map $\phi_{t}(\cdot;\omega_i)$ satisfies
$$P_x (\phi_{t}(S_i,V_i;\omega_i))\subset \mathbb{R}^{i-1}\times \mathbb{R}_+\times \mathbb{R}^{d-i}, \quad \forall t\geq 0,$$
$$ \lim_{t\to +\infty} P_x \left( \phi_{t}(x,v;\omega_i)\right)^{(i)}= 0, \quad \forall \ x\in S_i, \  v\in V_i.$$


Momentum ResNets are able to generate a contractive flow in all the components, see Figure \ref{fig:toolboxmr}. However, the difference with respect to the first order NODEs \eqref{NODE} is that for Momentum ResNets such contraction cannot be done in arbitrarily short time; this is the reason why Theorem \ref{sim:control} requires a minimal controllability time.

Nonetheless, it has to be mentioned that since the proof is constructive, it does not mean that such universal approximation cannot be obtained by other means with arbitrary short time.




\begin{corollary}\label{thm:ua}
Let $f\in L^{2}(\Omega;\mathbb{R}^d)$ with $\Omega \subset \R^d$ a bounded domain. Then, for any $\epsilon>0$,  there exist $T_{min}>0$, dependent on $\epsilon$, such that for all $T\geq T_{min}$ there exist controls $w,a\in L^\infty((0,T),\mathbb{R}^{d})$,  and $b\in L^\infty((0,T);\mathbb{R})$  whose associated flow according to \eqref{momnetI} satisfies
$$\|f(\cdot)-P_x\phi_T((\cdot,0);a,w,b)\|_{L^2(\Omega)}\leq \epsilon,$$
where $P_x$ is the projection to the state component of the system.
\end{corollary}

\subsection[Dynamical features]{Dynamical features of the Momentum ResNets} \label{s:toolbox}
{
\color{black}
Let us briefly describe the different types of flows that allow us to guarantee the results. For doing so, we will take controls $a,\  w\in \mathbb{R}^d$ and $b\in\mathbb{R}$; in this Subsection, we consider them to be time independent. More specifically, we will analyze a component at a time; fix $i\in\{1,...,d\}$ and consider $a$ in the form $a^{(k)}=\delta_{i,k}$, where $\delta_{i,k}$ is the Kronecker delta. By a simple translation, it is enough to understand the flows for $b=0$.
The ODE \eqref{momnetI} can be rewritten as: 
\begin{equation}
    \begin{cases}
    \ddot{x}+\dot{x}+w(t) \langle a(t),x\rangle =0 &\text{ if } \langle a(t),x\rangle>0\\
    \ddot{x}+\dot{x} =0&\text{ if } \langle a(t),x\rangle\leq 0    \end{cases}
\end{equation}
which component-wise is equivalent to
\begin{equation}\label{sys:ij}
    \begin{cases}
    \ddot{x}^{(j)}+\dot{x}^{(j)}+w^{(j)} x^{(i)} =0, &\text{ if } x^{(i)}>0,\\
    \ddot{x}^{(j)}+\dot{x}^{(j)} =0, &\text{ if } x^{(i)}\leq 0,    \end{cases}
\end{equation}


Of all possible cases, take into account these three:

\begin{enumerate}[leftmargin=5mm]
 \item $\boxed{\ddot{x}^{(i)}+\dot{x}^{(i)} =0}$ 
 Thus, we analyse the dynamics in the inactive hyperplane of the activation function.
Denoting $\dot{x}=p$, the solution is
\begin{equation}\label{dampedsol}
  \begin{array}{l}
    x^{(j)}(t)=x^{(j)}(0)+(1-e^{-t})(p^{(j)}(0)) \\
    p^{(j)}(t)=p^{(j)}(0)e^{-t}
   \end{array}
\end{equation}

 \item $\boxed{\ddot{x}^{(i)}+\dot{x}^{(i)}+w x^{(i)} =0 }$
    It corresponds to the case in which we are in the side of the space in which the activation function is not zero and $j=i$. The dynamics is
    \begin{equation*}
    \begin{pmatrix}
       x\\
       p
      \end{pmatrix}'=\mathcal{A}\begin{pmatrix}
      x\\
      p\end{pmatrix}=\begin{pmatrix}
      0 & 1\\
      -w & -1
      \end{pmatrix}\begin{pmatrix}
      x\\
      p\end{pmatrix},
      \end{equation*}
    where the eigenvalues and eigenvectors of the matrix $\mathcal{A}$ are given by
    $$\lambda_{\pm}=\frac{1}{2}\pm \sqrt{\frac{1}{4}-w},\ v_{\pm}=\begin{pmatrix}
               1\\
               -\frac{1}{2}\pm \sqrt{\frac{1}{4}-w}
              \end{pmatrix}.
    $$
    Here, three cases arise
      \begin{enumerate}
         \item if $w\in(0,1/4)$, the origin is an attractor; 
         \item if $w\geq 1/4$, the dynamics behaves as a damped oscillator;
         \item if $w<0$, the origin is a saddle point.
      \end{enumerate}
    \item $\boxed{\ddot{x}^{(j)}+\dot{x}^{(j)}+w^{(j)} x^{(i)} =0 }$ with $i\neq j$. Consider $w^{(j)}\neq 0$, 
    and assume that the quantity $q=w^{(i)}x^{(i)}$ is time independent. Then, the solution is
    {\small
    \begin{equation} \label{sol:parallel}
    \begin{split}
             x^{(j)}(t) &=x^{(j)}(0)+(1-e^{-t})p^{(j)}(0)\\
             &+q(t-1+e^{-t}),\\
         p^{(j)}(t) &= e^{-t}p^{(j)}(0)+q(1-e^{-t}).
    \end{split}
    \end{equation}  
    }


\end{enumerate}

}



\begin{figure}
    \centering
    \includegraphics[scale=0.22]{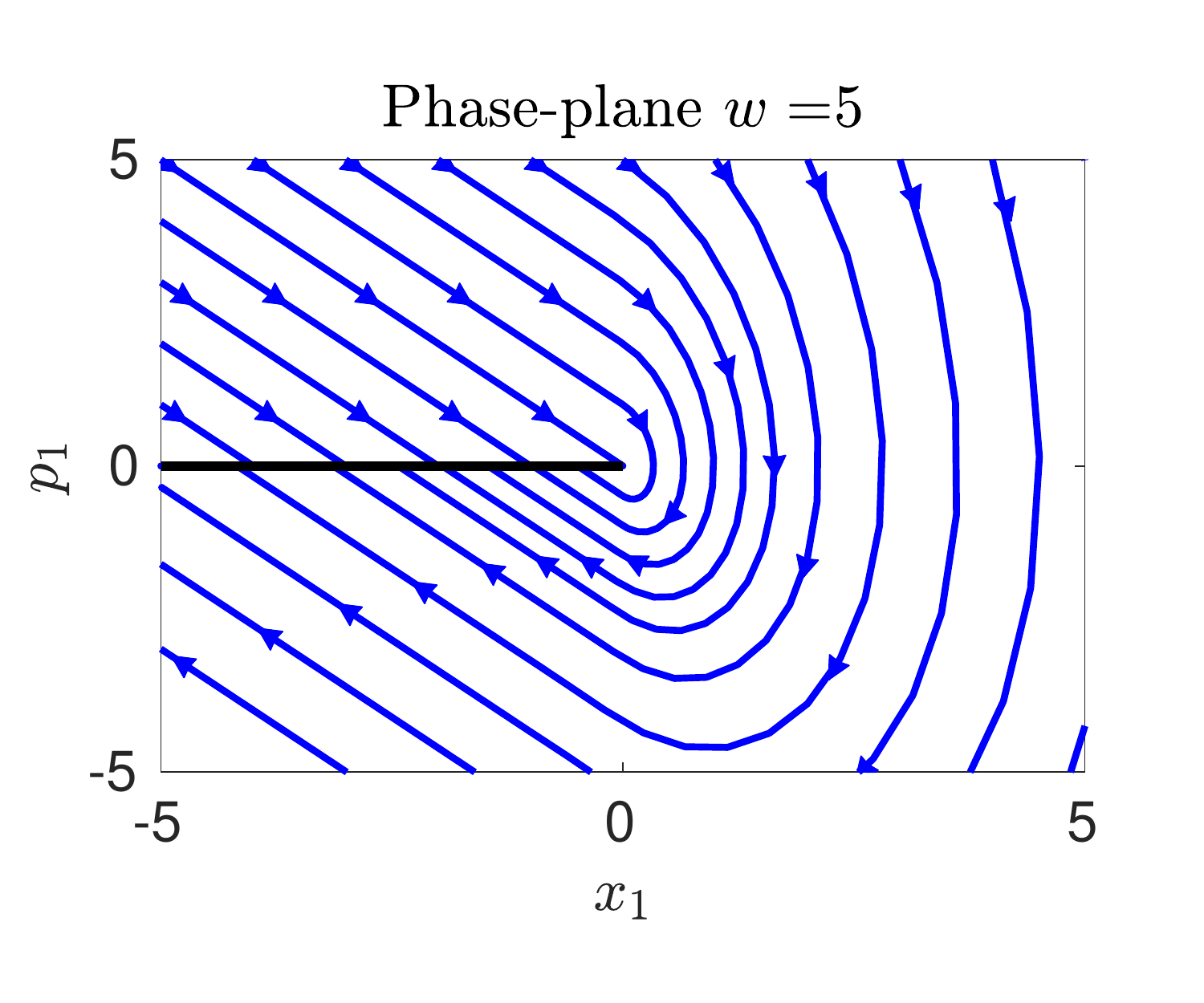}
        \includegraphics[scale=0.22]{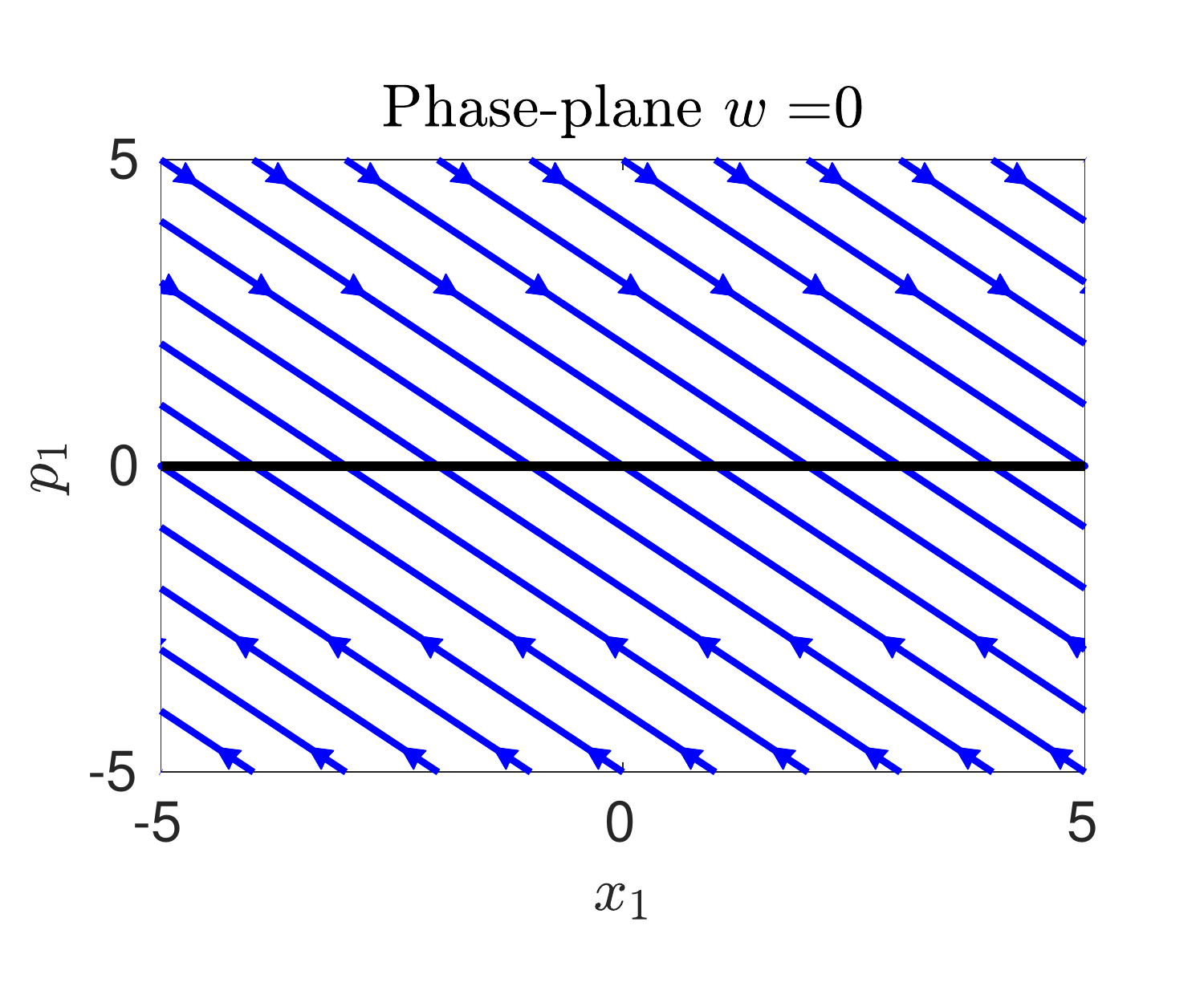}
        \includegraphics[scale=0.22]{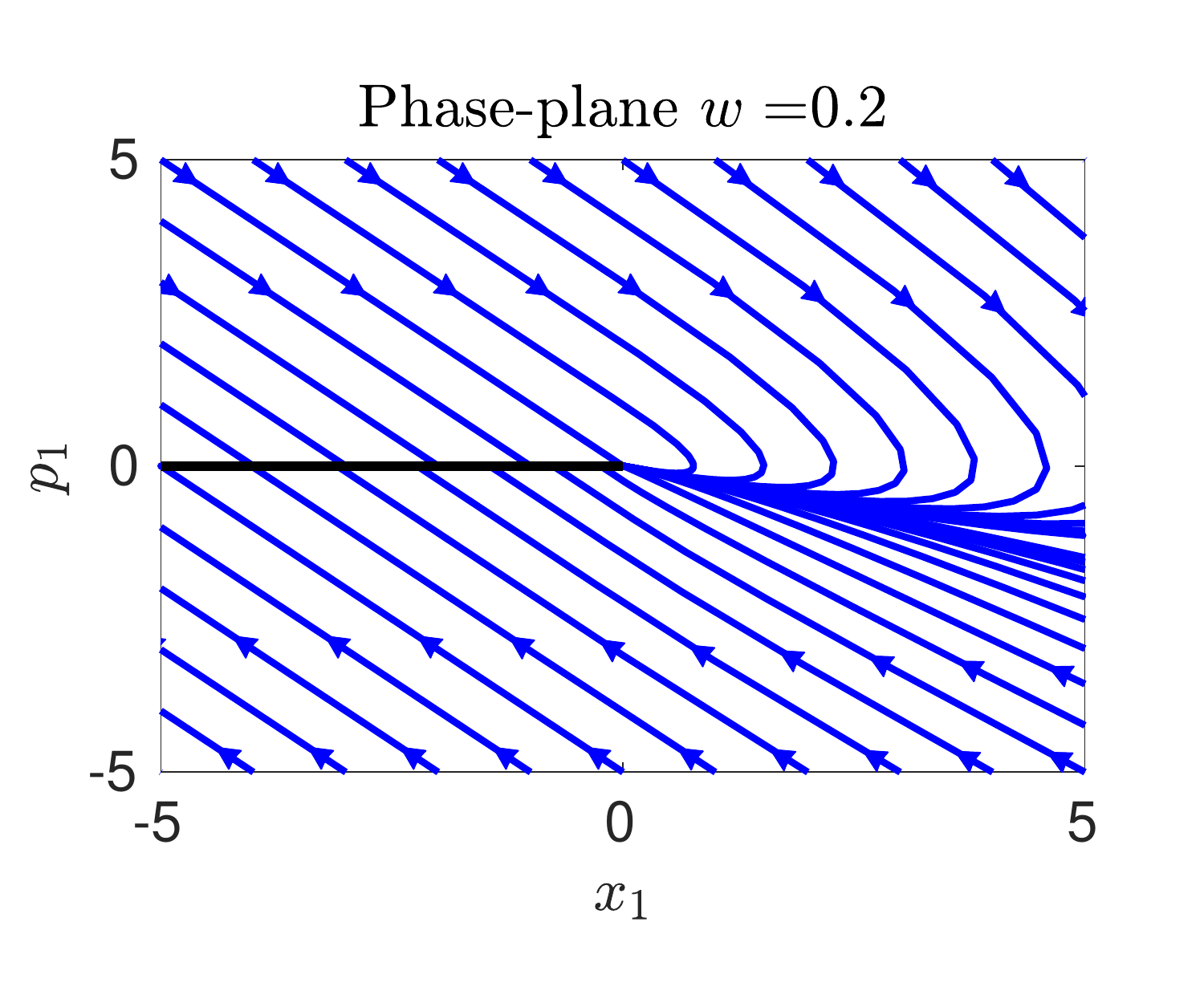}
    \caption{ {\footnotesize Phase plane in $x^{(1)}$ and $p^{(1)}$ for different flows, the black line are critical points for the dynamics. (Top-left): $w=5$ the half-oscillatory regime. (Top-right): $w=0$ damping regime. Bottom: $i=j$ with $w=0.2$ for $x_1>0$ (compressive flow).}}
    \label{fig:toolboxmr}
\end{figure}

\subsection{Ideas of the proofs}\label{ss:ideas}
{\color{black}
The control strategy will consist on splitting the time interval $(0,T)$ in two $(0,T/2)$ and $(T/2,T)$ and each interval is also divided in several subintervals and proceed interatively. In $(0,T/2)$ the goal will be to act appriopiately on the first component of each point whereas in $(T/2,T)$ the objective will be to act on the remaining components for each point. The argument is similar in both intervals, very schematically the procedure is: 
\begin{enumerate}[leftmargin=5mm]
    \item In each subinterval a specific point is chosen
    \item One chooses a control that will perturb the dynamics of several points with the goal of modifying appropriately the trajectory of the chosen point so that the free dynamics will bring the point to its target.
    \item We will make use of the fact that the control acts only in half space to not perturb the trajectories that have already been appropriately controlled.
\end{enumerate}
Setting $x(0)$ and $p(0)$ in \eqref{dampedsol} as the output of \eqref{sol:parallel} one arrives at:
 \begin{equation*}\begin{split}
     x^{(j)}(T)=x^{(j)}(0)+(1-e^{-t})p^{(j)}(0)+\\
     (1-e^{-(T-t)})(e^{-t}p^{(j)}(0))+\\
     q(t-1+e^{-t}+(1-e^{-(T-t)})(1-e^{-t}))
    \end{split}
     \end{equation*}
     One can realize that, for any $T>t>0$, the expression above is a line, and therefore, one can always choose $q$ so that at the time $T$ the state is at the desired position.

 \begin{figure}
    \centering
    \includegraphics[scale=0.5]{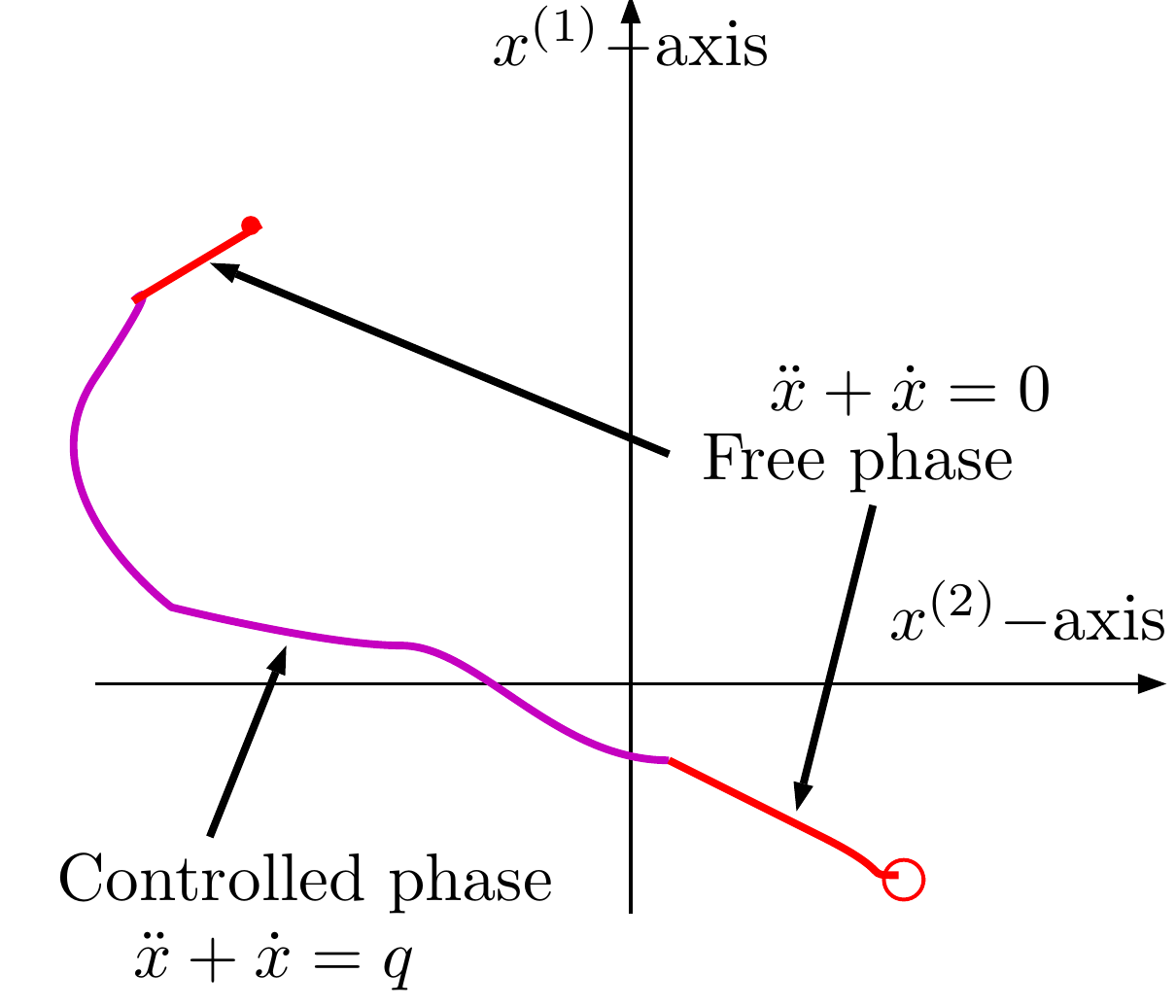}
    \caption{ {\footnotesize Representation of the strategy used for controlling points. The red point  (in the 2nd quadrant) represents the point one aims to control to the red circle (in the 4rth quadrant). In red, the free (uncontrolled) trajectory of the point when the control is not active in the region where the point is. In purple, the part of the trajectory where the control is active. The controlled phase ends before reaching the target, letting the point naturally reach the target following the free dynamics.}}
    \label{controlfree}
\end{figure}
}


The universal approximation follows with the same arguments than in \cite{ruizbalet2021neural}:
\begin{enumerate}[leftmargin=5mm]
 \item Consider a suitably fine mesh made out of hyperrectangles, and apply a compression to each rectangle.
 \item By Theorem \ref{sim:control}, we control a point of each compressed rectangle to its target. We can then conclude by continuity of the ODE with respect to the initial data.
\end{enumerate}
We point out that for $w\in (0,1/4)$ we can generate a compressive flow. One can see that using the eigenvectors of $\mathcal{A}$ one can define an invariant region inside $\{x^{(i)}\geq 0\}$ such that all the trajectories converge to the origin $(0,0)$, see Figure \ref{fig:toolboxmr}.

\section{A Neural ODE model}\label{ss:rnn}
\subsection[Simultaneous control]{Simultaneous control of the  state-memory pair}

Let $d,d_p\in \mathbb{N}$ and consider the following ODE system:
\begin{equation}\label{CRNN}
  \begin{cases}
   \dot{x}=w(t)\sigma\left(\langle a(t),x \rangle +\langle c(t), p\rangle +b(t)\right),\\
   \dot{p}=u(t)\sigma\left(\langle d(t),x\rangle +f(t)\right),\\
   x(0)=x_i,\quad p(0)=0.
  \end{cases}
\end{equation}
We will call $x$ the state component of the system and $p$ the memory component.

Our first goal is to be able to simultaneously control both the state and the memory at the final time for any target configuration with distinct state-memory pairs. The fact of being able to control also the memory will be key for the simultaneous tracking control problem that we will describe later on.

\begin{proposition}[Memory-Simultaneous controllability]\label{memorysimcontr}
 Let $d,d_p\in\mathbb{N}$, $T>0$ and
  consider $\{(x_i,p_i)\}_{i=1}^N\subset \mathbb{R}^d\times \mathbb{R}^{d_p}$ to be $N$ distinct initial data for \eqref{CRNN} and let us consider  $\{(y_i,\varphi_i)\}_{i=1}^N\subset \mathbb{R}^d\times \mathbb{R}^{d_p}$ to be distinct target points.
 Then, there exist controls $w,a,d,b_2\in L^\infty((0,T);\mathbb{R}^d)$,  $u,c\in L^\infty((0,T);\mathbb{R}^{d_p})$ and $b_1,f\in L^\infty((0,T);\mathbb{R})$ such that 
 $$ \phi_T((x_i,p_i);\omega)=(y_i,\varphi_i)\quad i\in\{1,...,N\} $$
 where by $\phi_T((x,p),\omega)$ we denote the solution at time $T$ of \eqref{CRNN} with initial data $(x,p)$ and controls $\omega=\{w,a,c,b_1,b_2,u,d,f\}$.
\end{proposition}

\begin{remark}[Approximate controllability]
 We have assumed that the targets are distinct. If two targets coincide (in state and memory), we cannot drive two points exactly there by the uniqueness of solution of the ODE. However, it is certainly possible to obtain an approximate controllability result, since we can control both points arbitrarily close to the common target by changing one of the targets.

\end{remark}

\begin{remark}[Complexity of the controls]
 The proof is constructive and inductive using piecewise constant controls. Therefore, the maximum required number of switches is of the order of $N$.
\end{remark}

\begin{remark}[Other activation functions]
One will notice that the only thing used from the activation function is that $\sigma$ is globally Lipschitz, $\sigma(x)=0$ if $x\leq 0$ and $\sigma(x)>0$ if $x>0$.
\end{remark}

The proof of Proposition \ref{memorysimcontr} makes use of an extension of the techniques developed in \cite{ruizbalet2021neural} for the control of 
\begin{equation}\label{NerRES}
 \dot{x}=w(t)\sigma\left(\langle a(t),x(t)\rangle+b(t)\right).
\end{equation}
In \cite{ruizbalet2021neural}, using Cartesian flows, the authors can prove a simultaneous control result for any $d\geq 2$. 

In this proof, we have two aspects to deal with. Firstly, notice that the field of the memory variable only depends on the state component, so, even if similar mechanisms to the case of first order the Neural ODE \eqref{NODE} can be applied, the intrinsic limitation of only depending on the state components should be discussed. Secondly, the dimensions $d$ and $d_p$ might not coincide. It would be certainly simpler to consider $d=d_p$, but as we shall see later on, the dimension of the memory plays a crucial role in the approximate simultaneous tracking controllability. 

Let us briefly remind some key features of the control of the Neural ODE for $d=2$ shown in \cite{ruizbalet2021neural}. The dynamics is
$$\begin{pmatrix}x^{(1)}\\ x^{(2)} \end{pmatrix}'
=\begin{pmatrix} w_1(t)\\ w_2(t)\end{pmatrix}\sigma(\langle a(t), x\rangle+b(t))
.$$
By using flows of these two types (see Figure \ref{previousflows}),
\begin{equation}\label{prev1}\begin{pmatrix}x^{(1)}\\ x^{(2)} \end{pmatrix}'
=\begin{pmatrix} w_1\\ 0\end{pmatrix}\sigma(a_2x^{(2)}+b)
\end{equation}
and
\begin{equation}\label{prev2}\begin{pmatrix}x^{(1)}\\ x^{(2)} \end{pmatrix}'
=\begin{pmatrix} 0\\ w_2\end{pmatrix}\sigma(a_1x^{(1)}+b)
\end{equation}
one can prove the approximate simultaneous controllability (see also Section \ref{ss:mainres} for Momentum ResNets).
As one can see from Figure \ref{previousflows}, the flows associated to \eqref{prev1} and \eqref{prev2} freeze one side of the hyperplane, selected by $\langle a,x\rangle +b=0$, and cause a parallel movement in the nonfrozen half space.

The main feature is that, in the equations in \eqref{prev1} and \eqref{prev2}, the  evolution of one component, say $x^{(1)}$ does not depend on $x^{(1)}$ itself, but on the other components. For this reason, even with the considered system  \eqref{CRNN}, the evolution of the memory component does not depend on itself, so we can still have a simultaneous controllability result. In the Appendix \ref{appsimcontr}, we will provide a proof of the simultaneous control result for the Neural ODE.

\begin{figure}
    \centering
    \includegraphics[scale=.27]{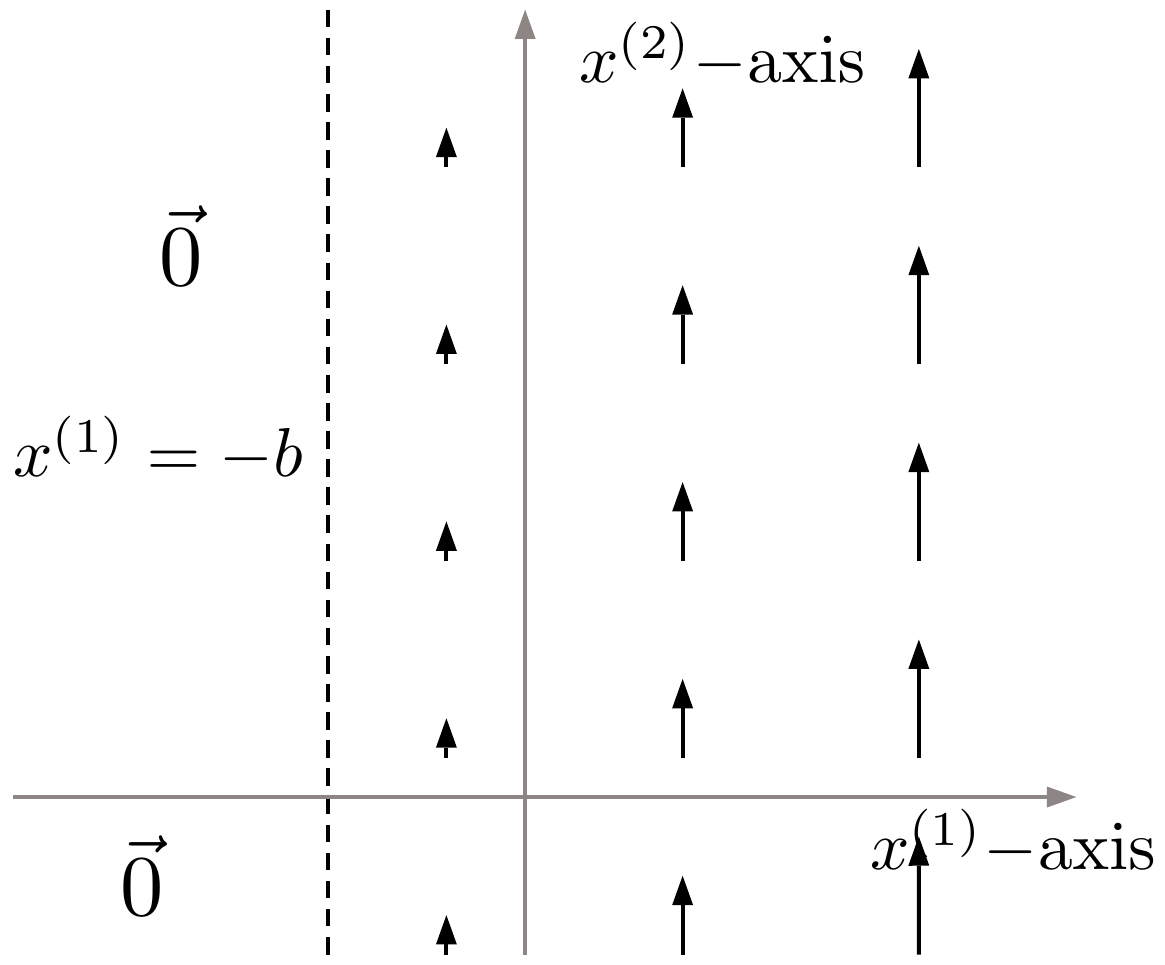}
        \includegraphics[scale=.27]{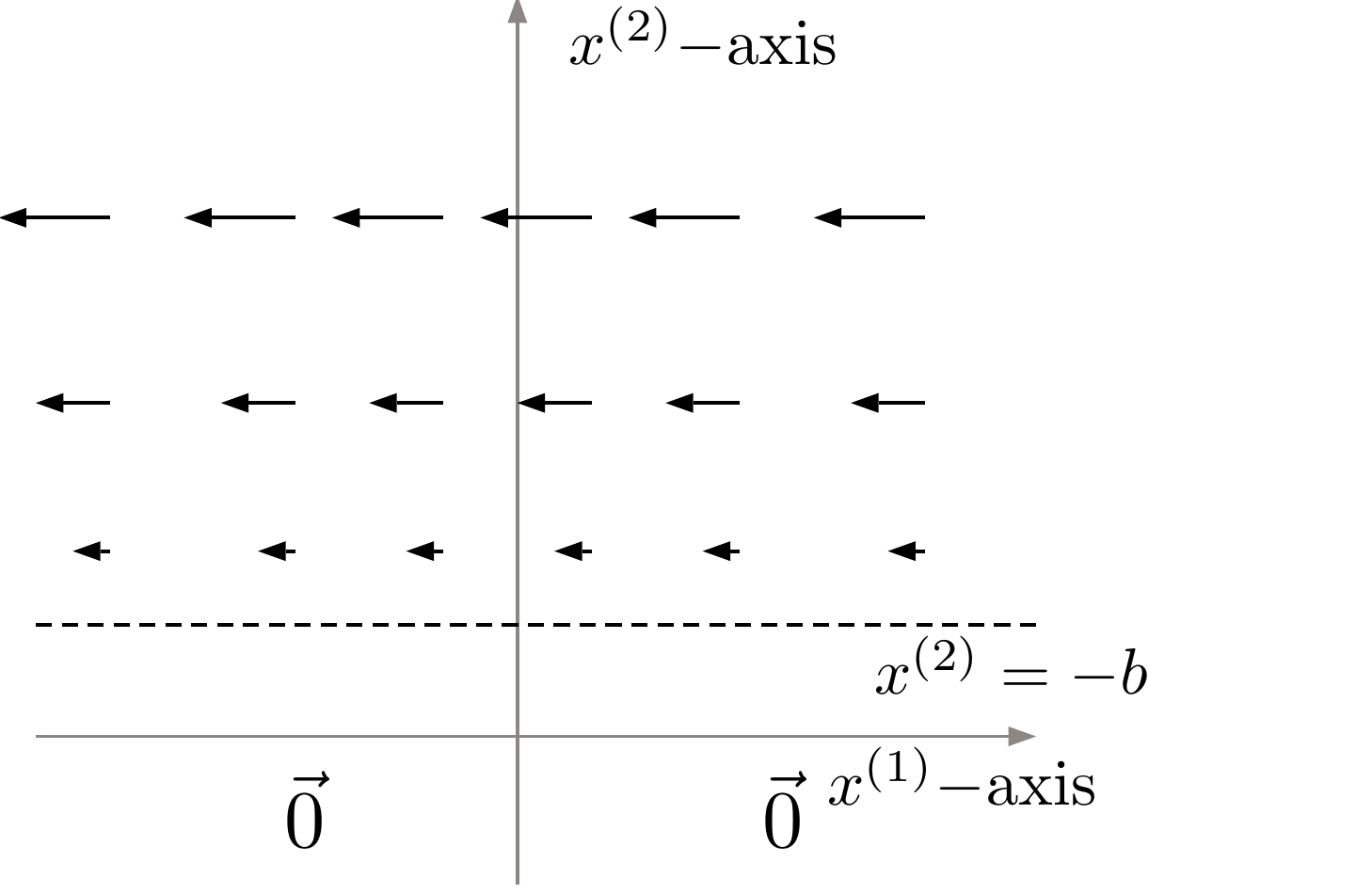}
    \caption{ {\footnotesize (Left) flow associated to \eqref{prev2}, (Right) flow associated to \eqref{prev1}}}
    \label{previousflows}
\end{figure}

The proof is postponed into the Appendix \ref{ApC}


\subsection{Simultaneous tracking control}
The goal of this section is to show how by the memory can help to obtain a simultaneous tracking controllability property.

We will consider a finite collection of samples of a continuous function $M$, 
$$M:\Omega\subset \mathbb{R}^d\mapsto C([0,T];\mathbb{R}^d)\cap BV([0,T];\mathbb{R}^d).$$
That is, we have a dataset of the form
{\small
$$\left\{(x_i,y_i=M(x_i) \right\}_{i=1}^N\subset \mathbb{R}^d\times C([0,T];\mathbb{R}^d)\cap BV([0,T];\mathbb{R}^d).$$
}
For being able to approximate the map $M$, we will consider an extension of \eqref{CRNN}. Instead of considering a scalar product between two vectors $\langle a,b\rangle$ or $\langle c,p\rangle$ in the state equation, we will consider a matrix vector product. 
Precisely, we take
\begin{equation}\label{CRNNE}
  \begin{cases}
   \dot{x}=W\boldsymbol{\vec{\sigma}}\left(Ax  +Cp +b_1\right)+b_2,\\
   \dot{p}=u\sigma\left(\langle d,x\rangle +f\right),\\
   x(-\tau)=x_i,\quad p(-\tau)=0.
  \end{cases}
\end{equation}
where $b_1,b_2\in L^\infty((-\tau,T);\mathbb{R}^{d}) $, \newline $W,A\in L^\infty((-\tau,T);\mathbb{R}^{d\times d})$,  $C\in L^\infty((-\tau,T);\mathbb{R}^{d\times d_p}) $. Note that the memory component of the system is the same as in \eqref{CRNN}. 
Again, notice that \eqref{CRNNE} can be rewritten as:
\begin{equation*}
   \dot{x}=W\boldsymbol{\vec{\sigma}}\left(Ax  +C\int_{-\tau}^tu\sigma\left(\langle d,x\rangle +f\right)ds +b_1\right)+b_2.\\
\end{equation*}

On the other hand, for suitable controls, the system \eqref{CRNNE} has the same flows as \eqref{CRNN}. In particular, Theorem \ref{memorysimcontr} also holds for \eqref{CRNNE} (and also the universal approximation of \cite{ruizbalet2021neural}).

The following theorem is the simultaneous tracking controllability, that we are able to achieve by taking the dimension of the memory as, at least, the double of the state dimension.

\begin{theorem}[Simultaneous Tracking controllability]\label{trackRNN}
Let  $d_p\geq d$ and let $\{(x_i,y_i)\}_{i=1}^N\subset \mathbb{R}^d\times \left(BV((0,T);\mathbb{R}^d)\cap C^0((0,T);\mathbb{R}^d)\right)$, and fix $\tau>0$. Then, for every $\epsilon>0$, there exist controls $W$, $A\in L^\infty((-\tau,T);\mathbb{R}^{d\times d})$, $C\in L^\infty((-\tau,T);\mathbb{R}^{d\times d_p}) $ and $d,b_1,b_2\in L^\infty((-\tau,T);\mathbb{R}^{d}) $, $u\in L^\infty((-\tau,T);\mathbb{R}^{d_p})  $ and $f\in L^\infty((-\tau,T);\mathbb{R})$ such that the solution of \eqref{CRNNE} satisfies:
 $$ \sup_{0\leq t\leq T}|\phi_t((x_i,0);\omega)-y_i(t)|<\epsilon \qquad \forall i\in\{1,...,N\},$$
 with
 $\omega=\{W,A,C,d,b_1,b_2,u,f\}$.
 \end{theorem}
In the interval $(-\tau,0)$, we use the controls to prepare the memory variables before the simultaneous tracking control process. We will postpone the proof in the Appendix \ref{ApC2}, giving only some flavour here.
 
 The proof of the theorem is based on the following two observations. We will consider $d_p=2d$; from here on, we will make an abuse of notation by writing $p=(p^{(1)},p^{(2)})$ where $p^{(1)},p^{(2)}\in\mathbb{R}^d$.
 
 \begin{enumerate}[leftmargin=5mm]
     \item We begin by approximating linear maps of the form $y_i=P_it+B_i$ with $P_i$, $B_i\in \mathbb{R}^d$. 
     For every $i\in\{1,...,N\}$ let $B_i$ be the targets for the state for Theorem \ref{memorysimcontr} and $p_i$ the targets for the memory to be found hereafter.
     Then, consider $A=0$, $b_1=0$, $W=I_d$; the dynamics of the state equation is
     $$x'=\boldsymbol{\vec{\sigma}}(Cp)+b_2.$$
     We choose $C=(I_d|0)$; 
     since we consider $\sigma$ to be the ReLU, we reduce the problem to finding $b_2\in \mathbb{R}^d$ such that
     $$p_i^{(1)}=P_i-b_2\in \mathbb{R}_{+}^d\qquad i\in \{1,...,N\},$$
     which is always possible. Then the dynamics reads
          $$ x'=P_i,\quad x(0)=B_i.$$
          The solutions to the latter are precisely $P_it+B_i$ for all $i\in\{1,...,N\}$.
    \item The second point is to use the first and second component of the memory in an alternate manner. We will first approximate the target functions by means of piecewise linear functions of time, as Figure \ref{fig:piecwise} shows. Then, we will use an alternate strategy, while we are using the first component of the memory to control the system, we will reconfigure the second component to be prepared for the next interval as Figure \ref{fig:memdiag}.
    
    \begin{figure}
        \centering
        \includegraphics[scale=0.4]{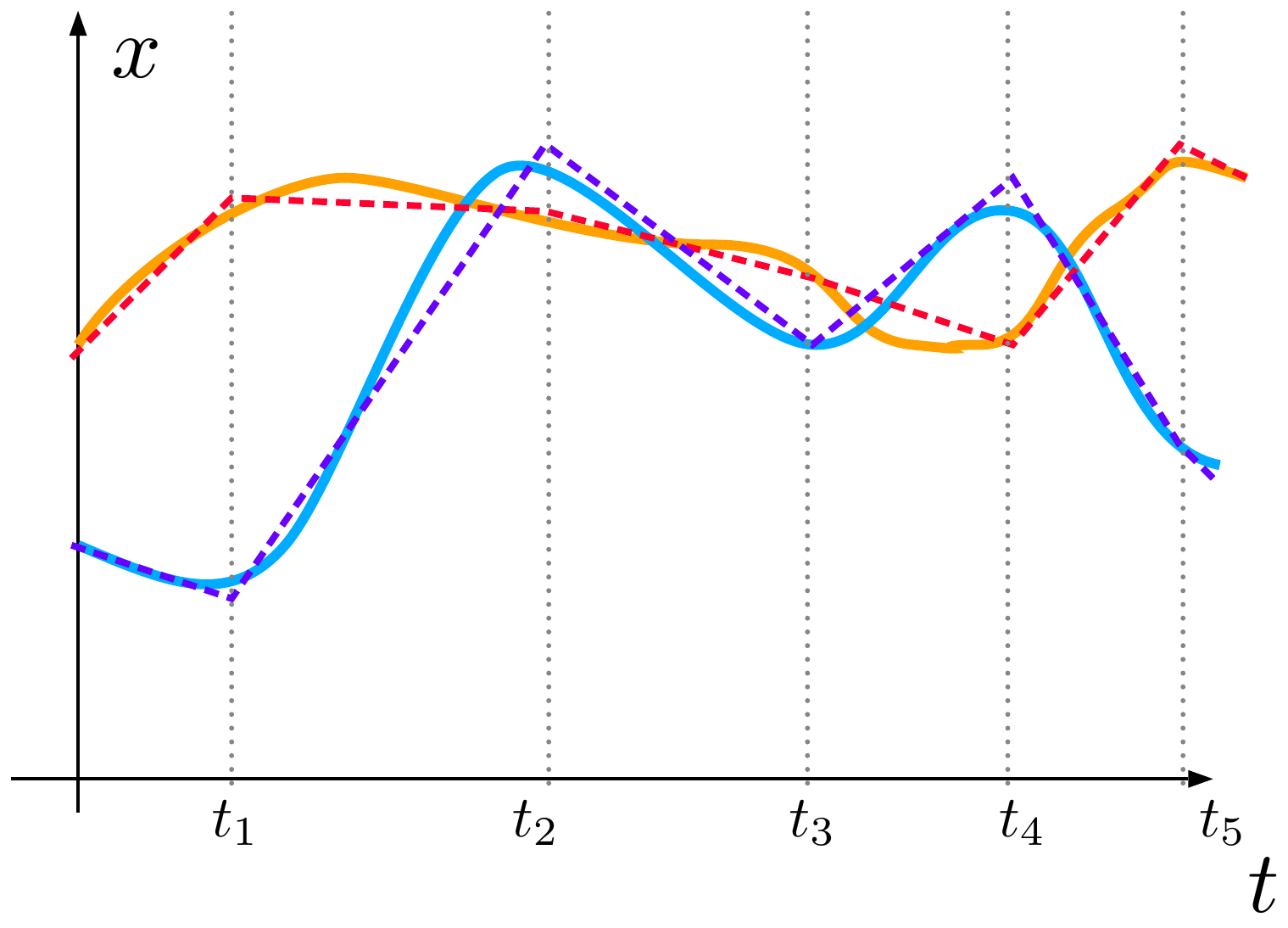}
        \caption{ {\footnotesize Piecewise linear approximation of the target functions. The solid lines represent the target functions, while the dotted coloured lines represent its piecewise linear approximation. The approximation generates time intervals $(t_k,t_{k+1})$ where the memory will be used to control the trajectory. }}
        \label{fig:piecwise}
    \end{figure}
    
        \begin{figure}
        \centering
        \includegraphics[scale=0.35]{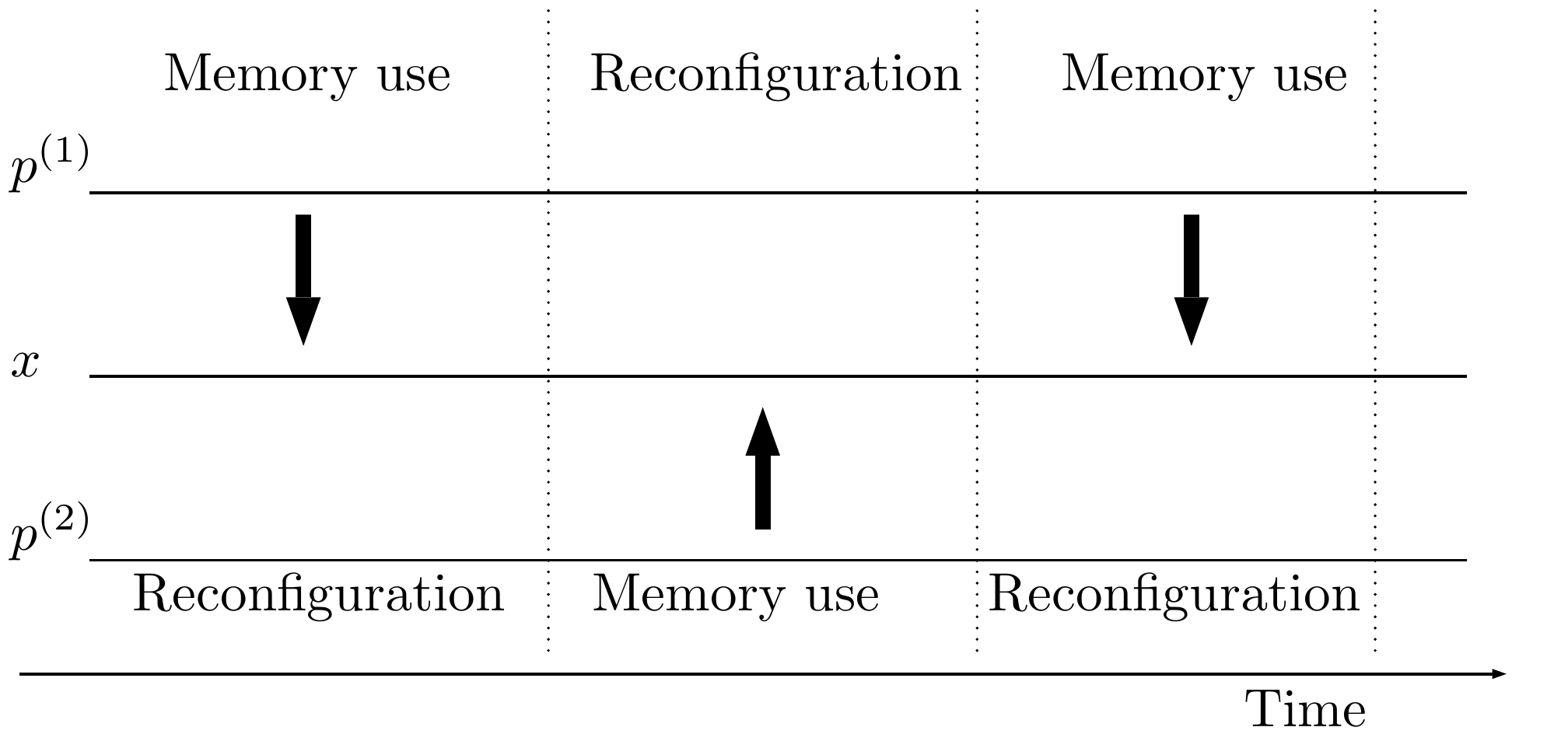}
        \caption{ {\footnotesize Qualitative representation of the alternate control strategy, while we are using $p^{(1)}$ to approximate in a time interval (Memory use), we are preparing $p^{(2)}$ for the next time interval (Reconfiguration). In the successive time interval the roles are interchanged, we use $p^{(2)}$ for controlling while $p^{(1)}$ is being prepared for the next time interval.}}
        \label{fig:memdiag}
    \end{figure}

 \end{enumerate}
 
\begin{remark}[Other activation functions]
In the proof of Theorem \ref{trackRNN} we have strongly used the structure of the ReLU activation function. However, similar constructions can be done for activation functions that satisfy
\begin{align*}
\sigma(x)=0\qquad x\leq0,\quad
\sigma(x)>0\qquad x>0
\end{align*}
with $\sigma$ being differentiable from the right, i.e.
$$\exists \lim_{h\to 0^+} \frac{\sigma(h)}{h}=:\sigma'_+(0).$$
Roughly speaking, the property that one can use is that around $0$ the activation function is like the ReLU
\begin{equation*}
    \sigma(x)\approx \max\{0,\sigma'_+(0)x\} \quad \text{if } |x|<<1 .
\end{equation*}
The proof is adapted by simply choosing small target memories in the positive cone, i.e.  $p_i\in \mathbb{R}^d_+\cap \mathbb{B}_\delta(0)$ for $\delta>0$ small enough and by choosing $W=\alpha I_d$  for $\alpha\in\mathbb{R}$ big enough so that the memory component that needs to be used, say $p_{i,k}^{(1)}$, fulfills
$$\alpha p_{i,k}^{(1)}=P_{i,k}-b_{2,k}$$
where $P_{i,k}-b_{2,k}$ are as in the proof of Theorem \ref{trackRNN}.
\end{remark}

\begin{corollary}[Universal Tracking Approximation]\label{cor:uta}
Fix $d_p\geq 2d$. Let $\Omega\subset \mathbb{R}^d$ be a bounded set and  $M\in C^0\left(\Omega; C^0((0,T);\mathbb{R}^d)\cap BV((0,T);\mathbb{R}^d)\right)$. Fix $\tau>0$, then for every $\epsilon>0$ there exist controls $\omega$ such that the solution of \eqref{CRNNE} satisfies:
$$ \sup_{t\in(0,T)}\left\|M(\cdot,t)-P_x\phi_t(\cdot;\omega)\right\|_{L^2(\Omega)}<\epsilon, $$
where $P_x$ is the projection into the state variables.
\end{corollary}
\begin{proof}
One can observe that, using the compression Lemma 1 of \cite{ruizbalet2021neural} in the universal approximation for the state component, and then the simultaneous tracking control Theorem \ref{trackRNN}, we have a universal tracking approximation for functions in $C^0\left(\Omega; C^0((0,T);\mathbb{R}^d)\cap BV((0,T);\mathbb{R}^d)\right)$ for $\Omega\subset\mathbb{R}^d$ bounded. 
 \end{proof}

\section{Simulations}\label{ss:simus}
This section is devoted to illustrating through simulations the generalization properties discussed above of the models considered.

We will consider the two types of problems studied in the paper
\begin{enumerate}[leftmargin=5mm]
    \item Learning a function;
    \item Simultaneous tracking controllability.
\end{enumerate}

\subsection{Learning a function}
For the first problem, we choose a function $$g:[-1,1]^2\to \mathbb{R},$$ precisely the characteristic function of a circle centred in $(0,0)$. We take $N=100$ random samples of it. Then, we use these samples to train a Neural ODE (see Figures \ref{fig:Ng1} and \ref{fig:Ne1}) and a Momentum ResNet (see Figure \ref{fig:Mg1} and \ref{fig:Me1}) and observe the differences.

We employ a least-squares approach by minimizing
$$J(\{x_i,y_i\}_{i=1}^N)=\frac{1}{N}\sum_{i=1}^N\|Px_i(T;\omega)-y_i\|_{L^2}^2+\beta \|\omega\|_{L^2}^2, $$
where by $\omega$ we understand the controls as well as the projection $P$, i.e. $\omega=(a,w,P,b)$ belongs to $$\mathcal{X}:= (L^\infty((0,T); \mathbb{R}^2))^3\times L^\infty((0,T); \mathbb{R})$$
 $x_i(T;\omega)$ is the solution at time $T$ of either a Momentum ResNet as in \eqref{eq:resnet} with initial datum $(x_i,0)$ 
or a  Neural ODE 
\begin{equation*}
    \begin{cases}
    x'=w(t)\sigma(\langle a,x\rangle +b),\\
    x(0)=x_i.
    \end{cases}
\end{equation*}
For both models, we take a sigmoid activation function.

Being $\mathcal{S}$ a set and $\chi_\mathcal{S}$ the characteristic function of such set, we measure the generalization error by means of the $L^1$ norm, i.e.
$$\text{error}=\|P\phi_T(x;\omega)-\chi_{\mathcal{S}}\|_{L^1}, $$
where with an abuse of notation we denote $P \phi_T$ both the solution of the Neural ODE and the projection to the position of the solution of the Momentum ResNet. 

At the final time,  a suitable hyperplane $H$ separates the training points, determining whether their image is 0 or 1. That is, this hyperplane fix the decision boundary. Since the map 
$$ \phi_T:\Omega\to \mathbb{R}^d$$
is a homeomorphism, we have that $(\phi_T)^{-1}(H)$ is topologically equivalent to a hyperplane. 
In particular, the decision boundary for a first-order NODE cannot have the same topology of the boundary of the circle, which is a closed curve. 
So, the decision boundary touches the border of the square in Figure \ref{fig:Ng1}, misrepresenting some points that are far from the border of the circle.

On the other hand, for Momentum ResNets, the decision boundary is determined by a projection only on the state variable. 
This means that  the decision boundary is not necessarily topologically equivalent to a hyperplane. 
From the simulations, we see that the Momentum ResNet has been able  to capture the topology of the set $\mathcal{S}$ and to perform the task with smaller error.

\begin{figure}
    \centering
    \includegraphics[scale=0.3]{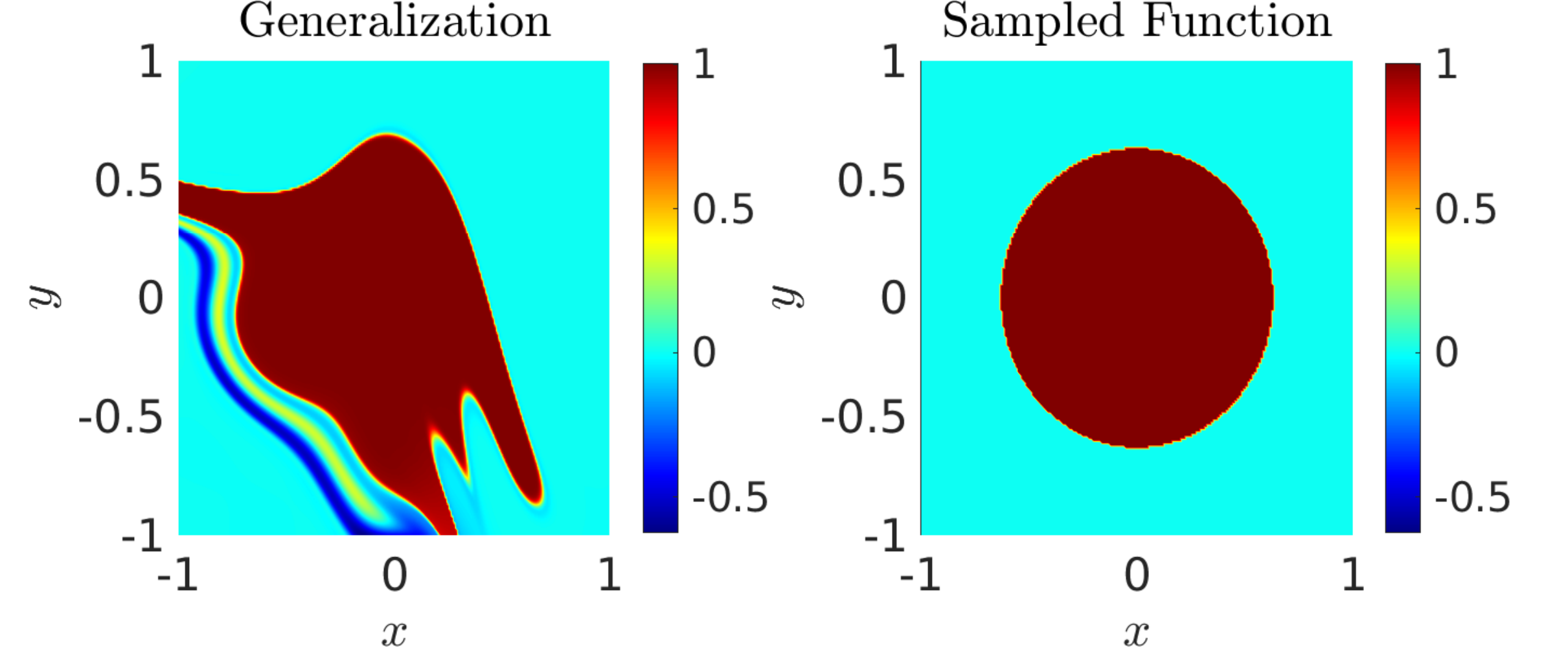}
    \caption{ {\footnotesize Neural ODE with trained with 100 data points sampled in a disk, 25 layers for the time discretization with $T=10$, $\beta=10^{-6}$. }}
    \label{fig:Ng1}
\end{figure}
\begin{figure}
    \centering
    \includegraphics[scale=0.4]{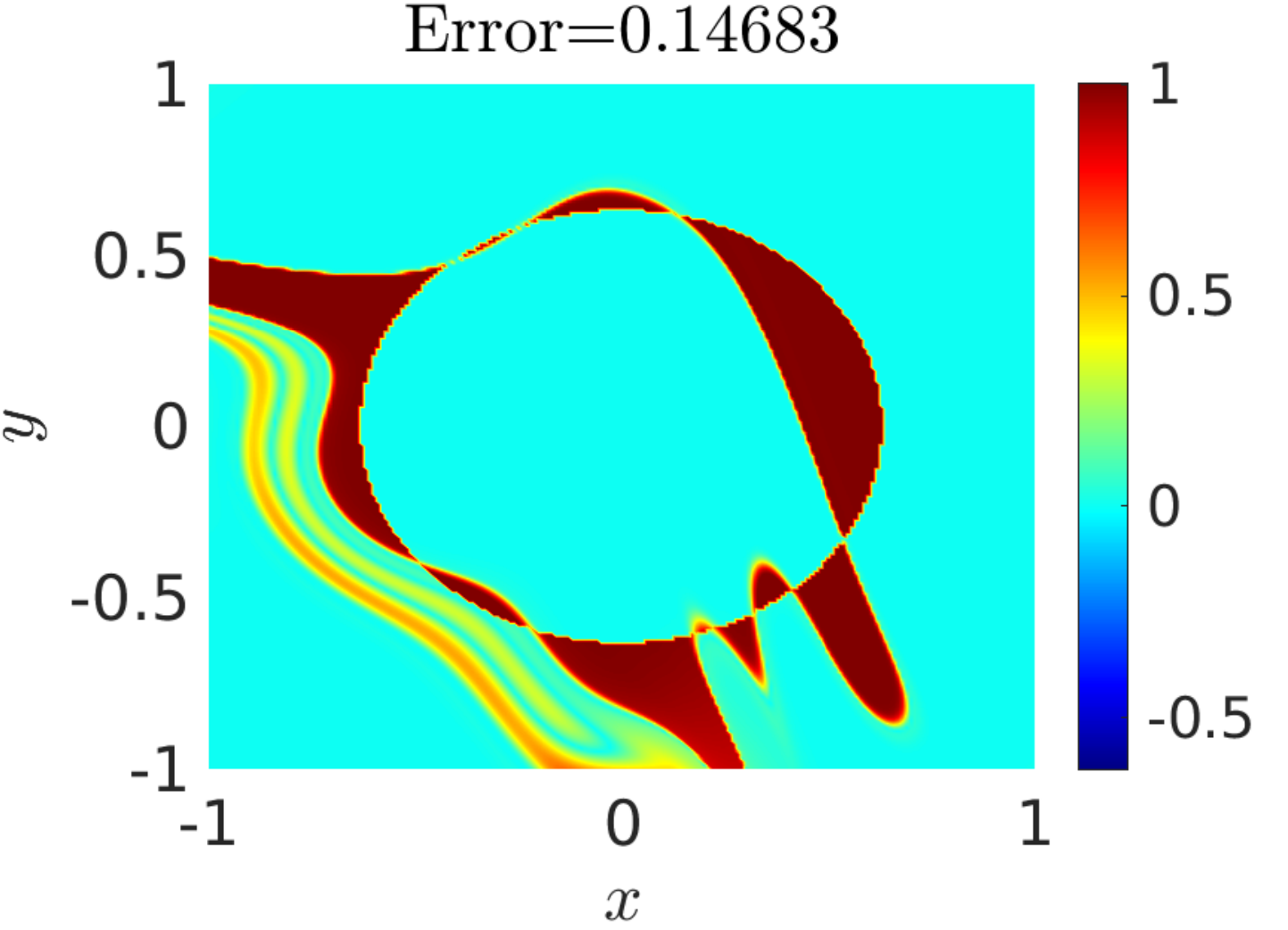}
    \caption{ {\footnotesize Error plot for the Neural ODE with trained with 100 data points sampled in a disk, 25 layers for the time discretization with $T=10$, $\beta=10^{-6}$.}}
    \label{fig:Ne1}
\end{figure}
\begin{figure}
    \centering
    \includegraphics[scale=0.3]{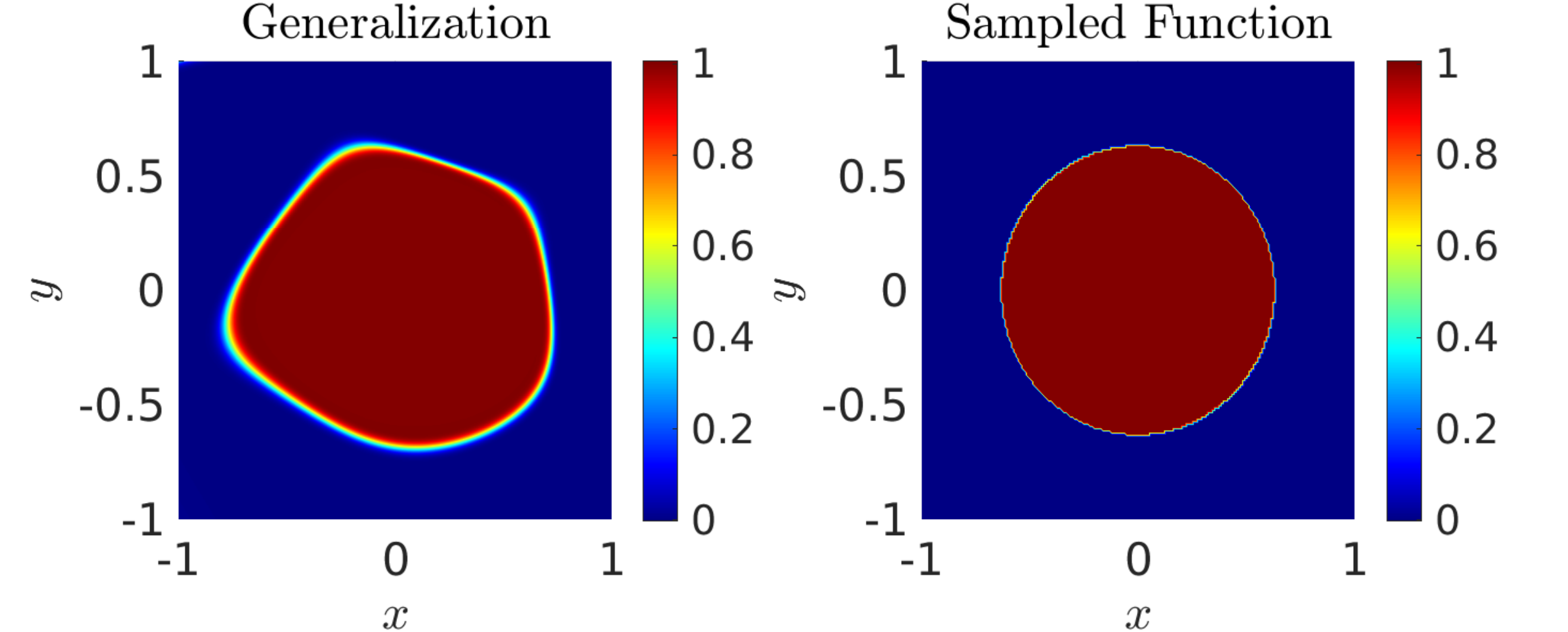}
    \caption{ {\footnotesize Momentum Resnet with trained with 100 data points sampled in a disk, 25 layers for the time discretization with $T=10$, $\beta=10^{-6}$.}}
    \label{fig:Mg1}
\end{figure}
\begin{figure}
    \centering
    \includegraphics[scale=0.4]{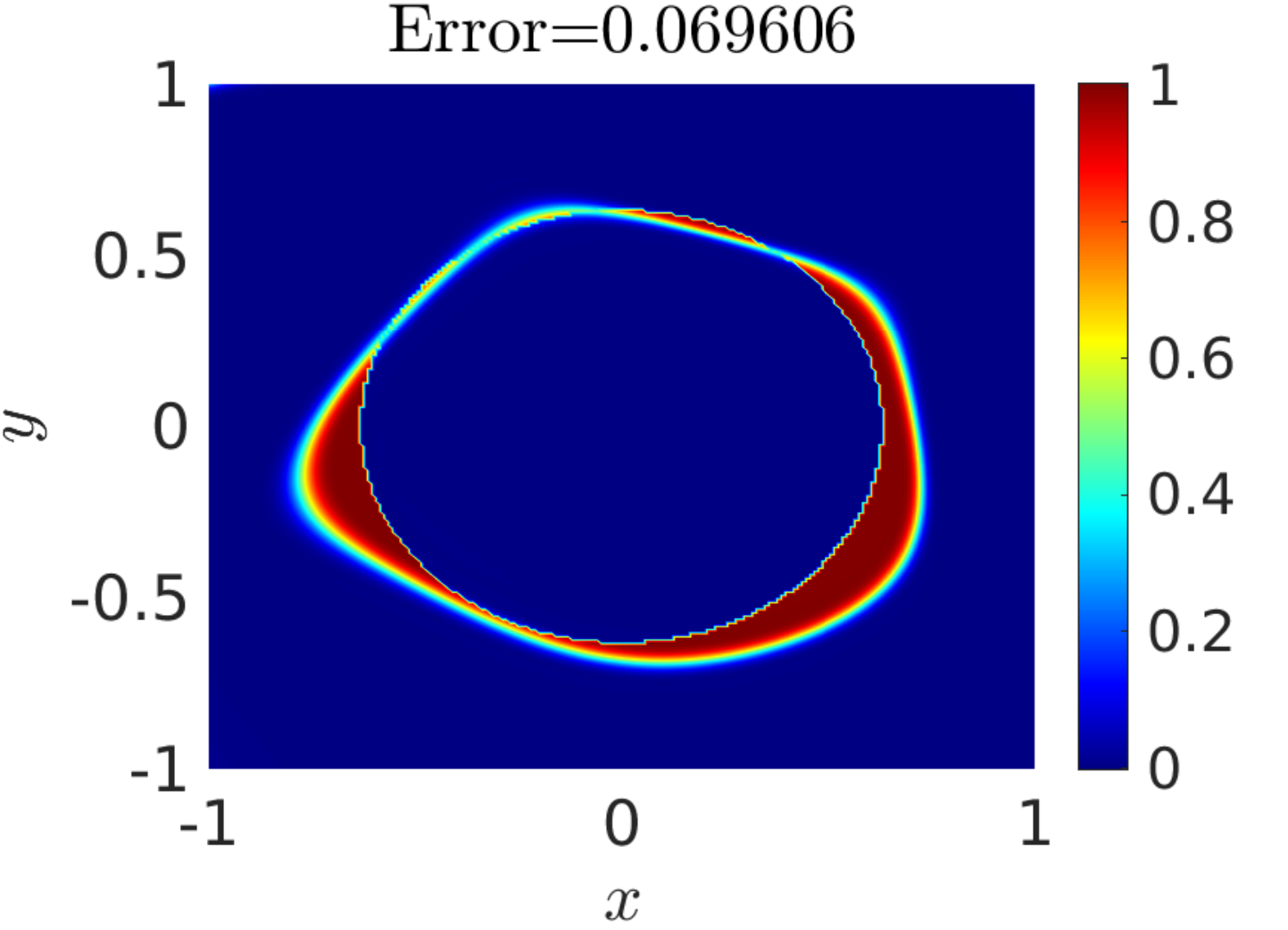}
    \caption{ {\footnotesize Error plot for the Momentum ResNet with 100 data points sampled in a disk, 25 layers for the time discretization with $T=10$, $\beta=10^{-6}$.}}
    \label{fig:Me1}
\end{figure}

\subsection{Simultaneous tracking}
We will consider a function 
$$M\in C^0\left((0,1);C^0((0,T);\mathbb{R})\cap BV((0,T);\mathbb{R})\right),$$
precisely $M(x)=\sin(xt)$. We take a sample of $5$ random points $x_i\in (0,1)$ and define 
$$y_i=M(x_i).$$
We minimize the functional
\begin{equation*}\begin{split}
   J\left(\{x_i,y_i\}_{i=1}^N\right)=\frac{1}{N}&\sum_{i=1}^N\int_0^T\|x_i(t;\omega)-y_i(t)\|_{L^2}^2dt\\
   &+\beta \|\omega\|_{\mathcal{X}}^2, 
  \end{split}
\end{equation*}
where $x_i(t;\omega)$ is the solution of
\begin{equation*}
\begin{cases}
    x'=w\sigma( a x + \left\langle c,\int_{-\tau}^t u\sigma(dx +f)ds\right\rangle +b_1)+b_2\\
    x(-\tau)=x_i,
    \end{cases}
\end{equation*}
with $a,\ w, \ d$, $b_1, \ b_2\in L^\infty((-\tau,T);\mathbb{R})$ and $c, \ u\in L^\infty((-\tau,T);\mathbb{R}^2)$.

We get an error smaller than 0.063529 for the simultaneous tracking and an error of 0.11019 for the overall universal approximation.

\begin{figure}
 \includegraphics[scale=0.25]{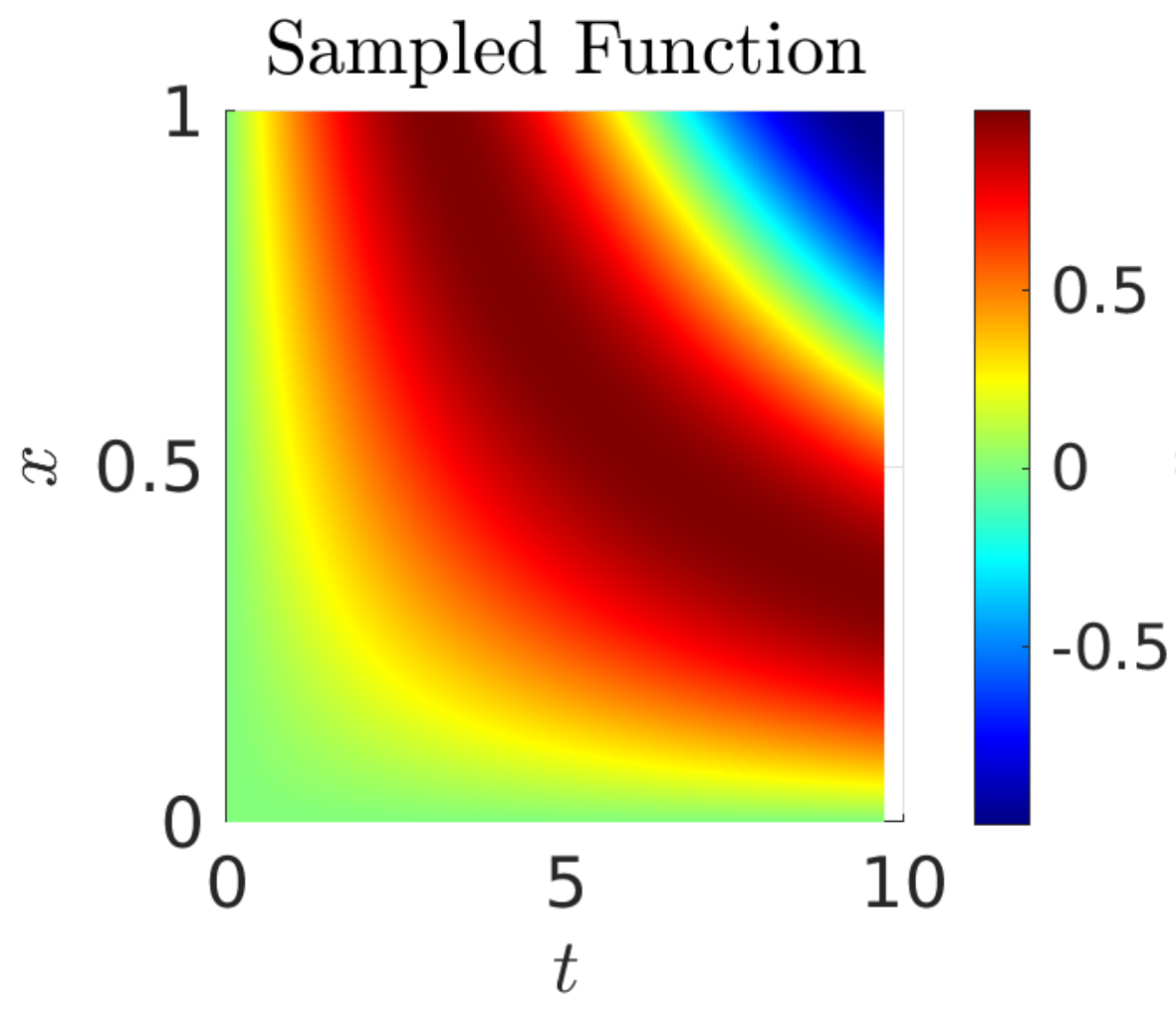}
  \includegraphics[scale=0.25]{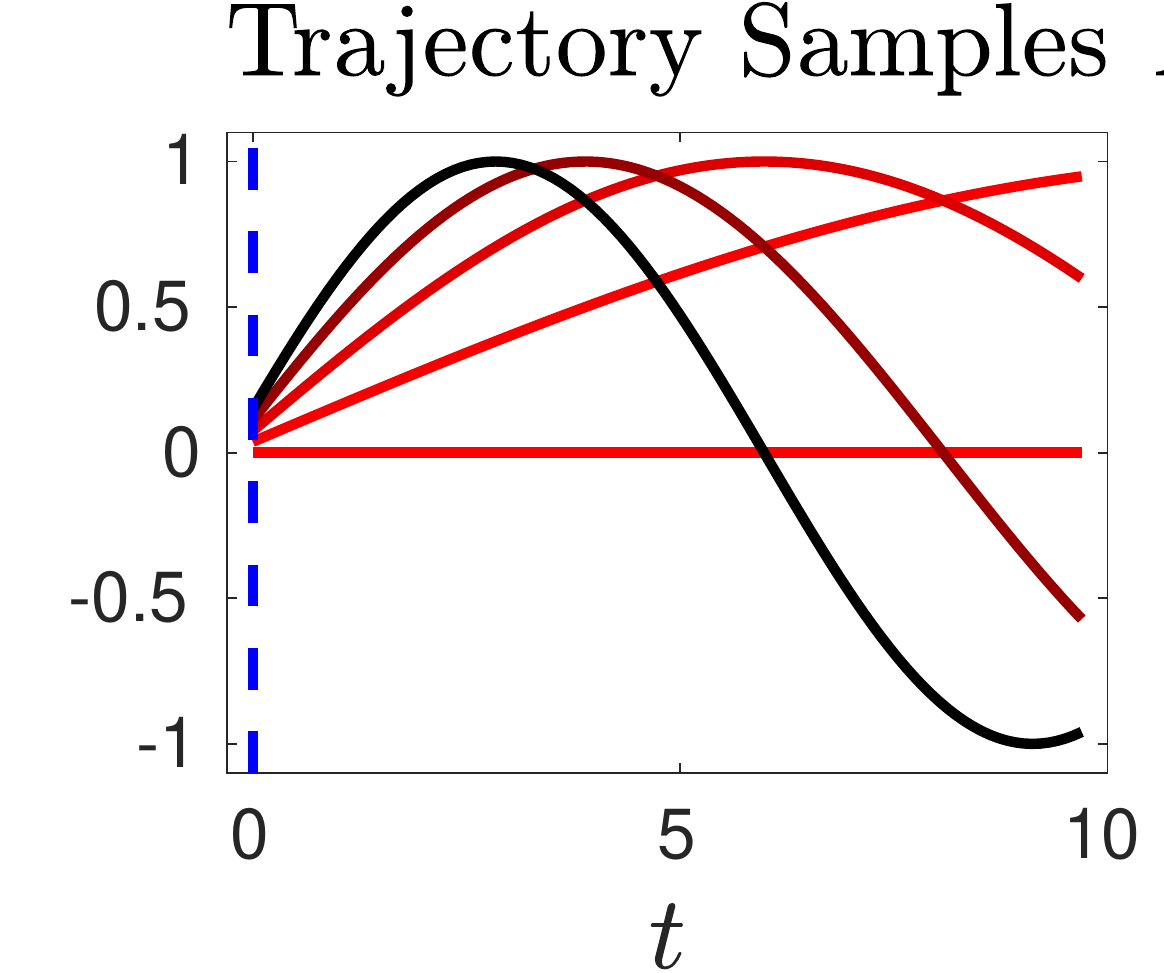}
  
   \includegraphics[scale=0.25]{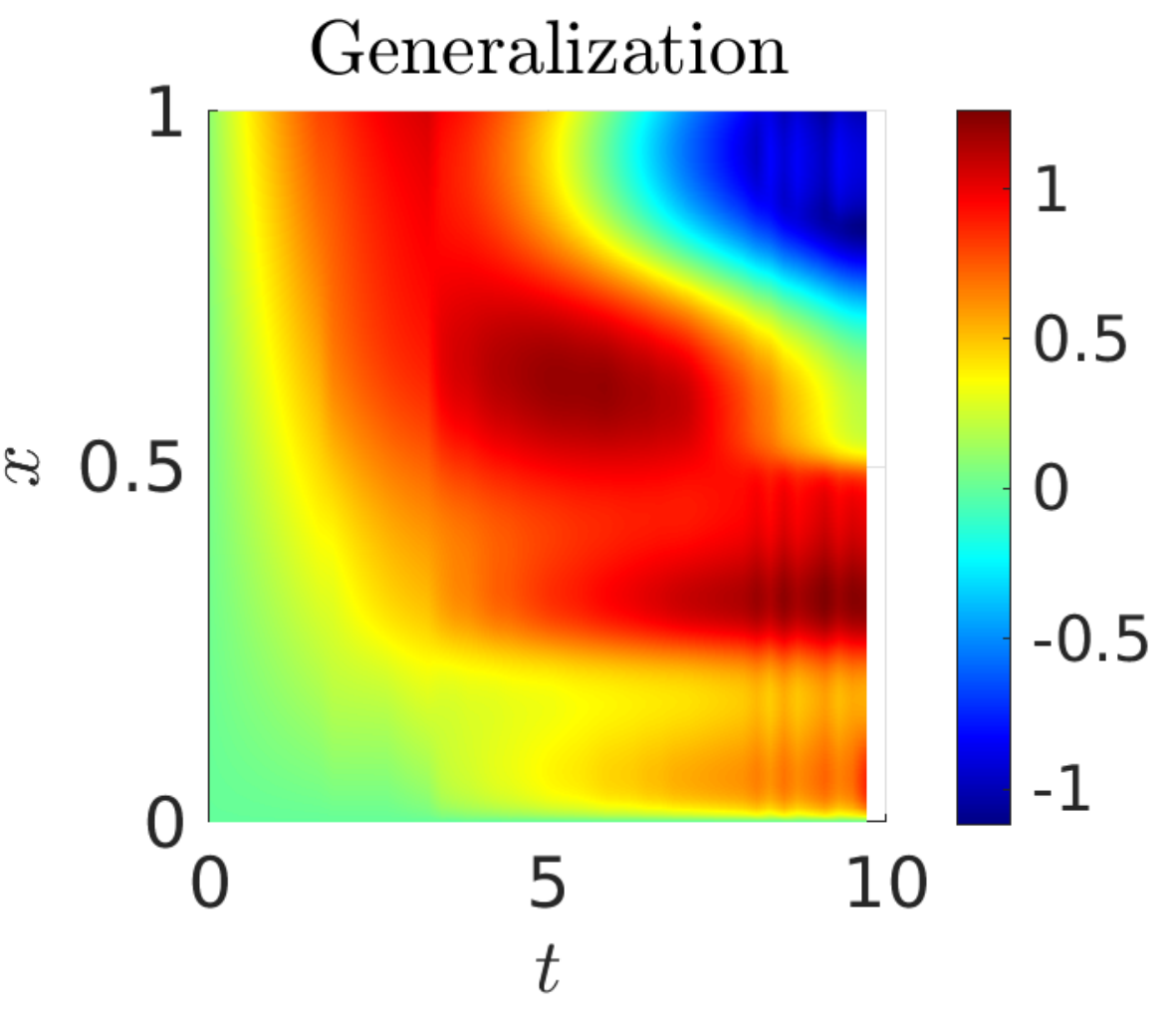}
    \includegraphics[scale=0.25]{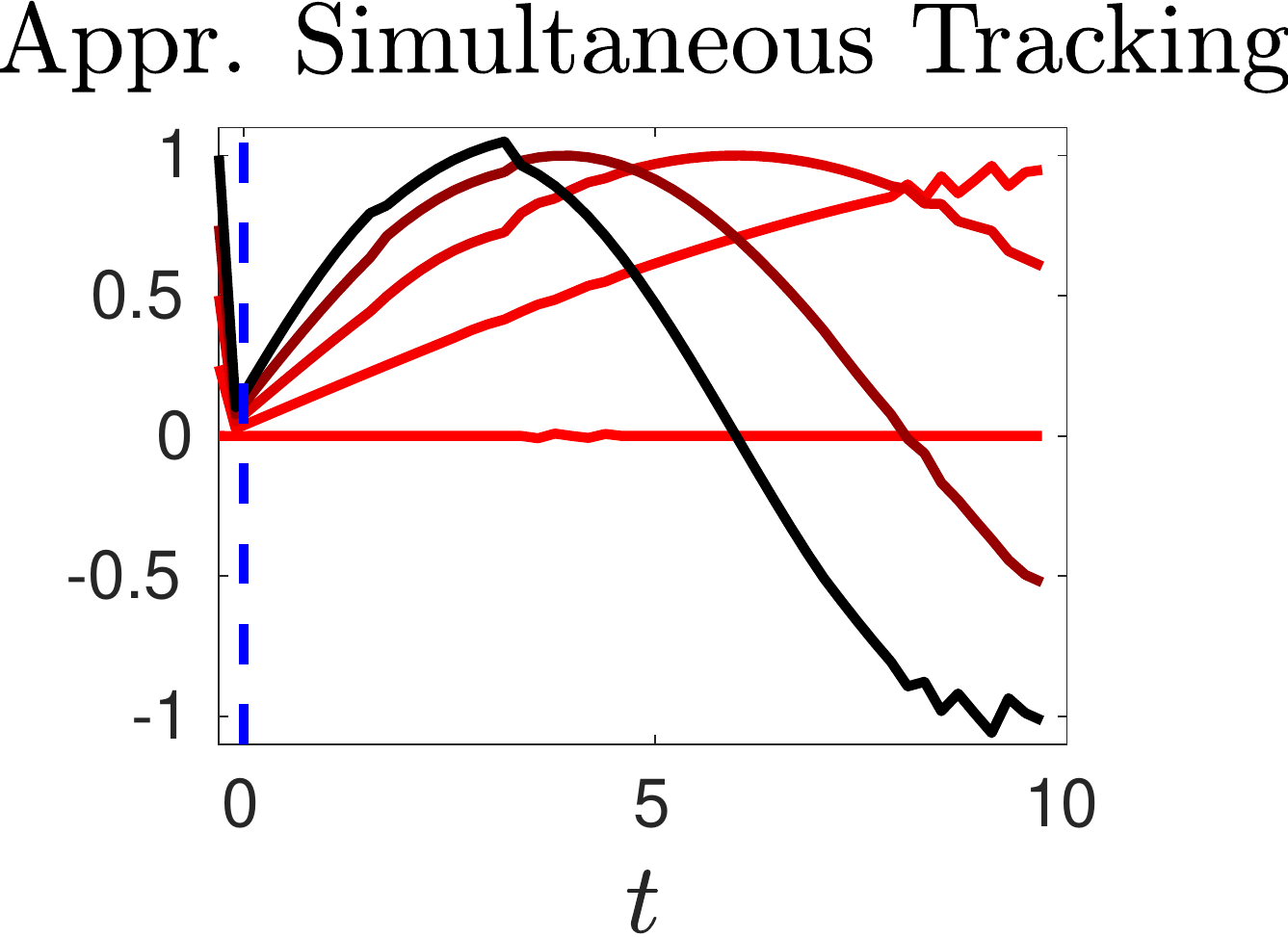}
    
  \includegraphics[scale=0.2]{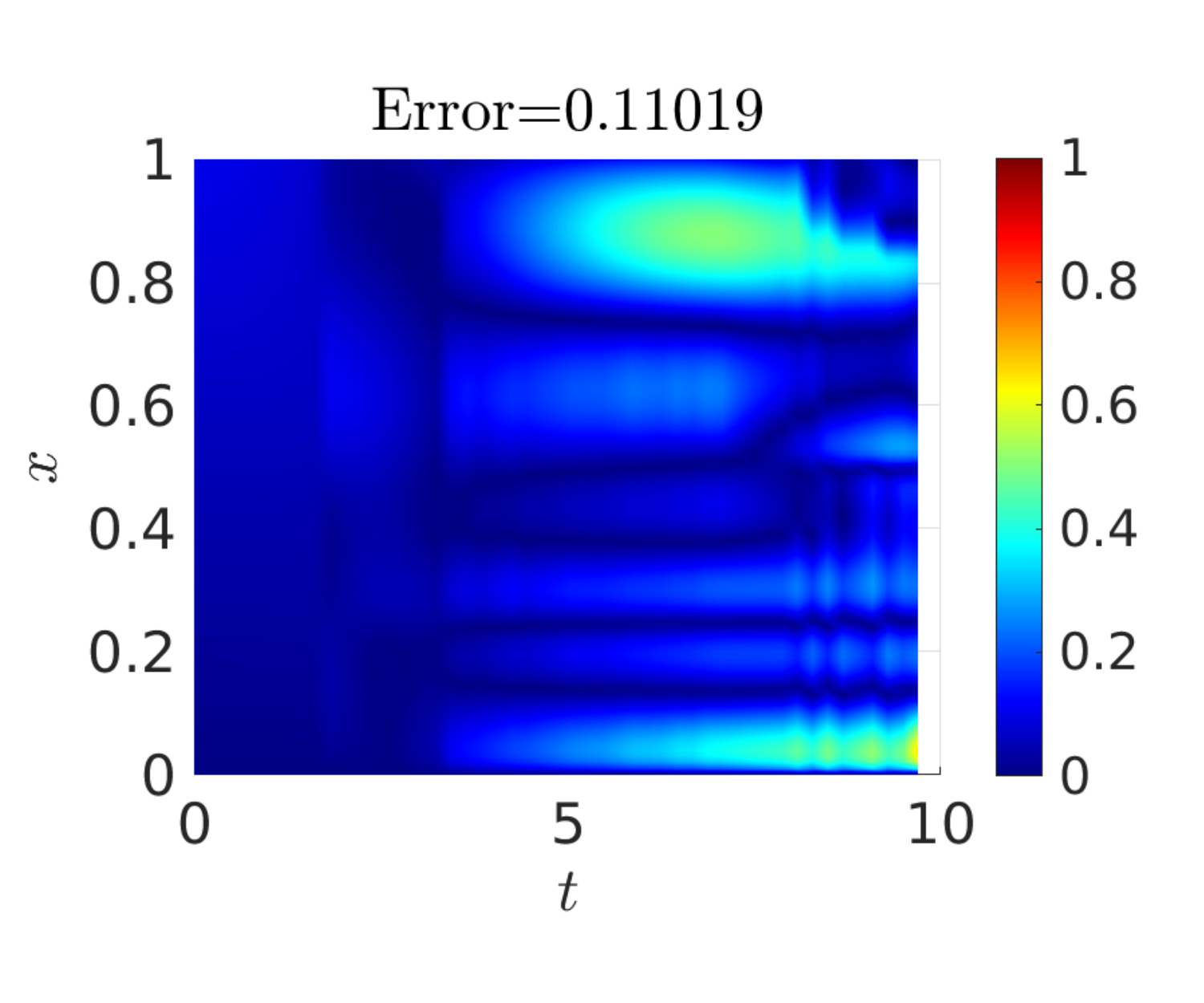}
    \includegraphics[scale=0.2]{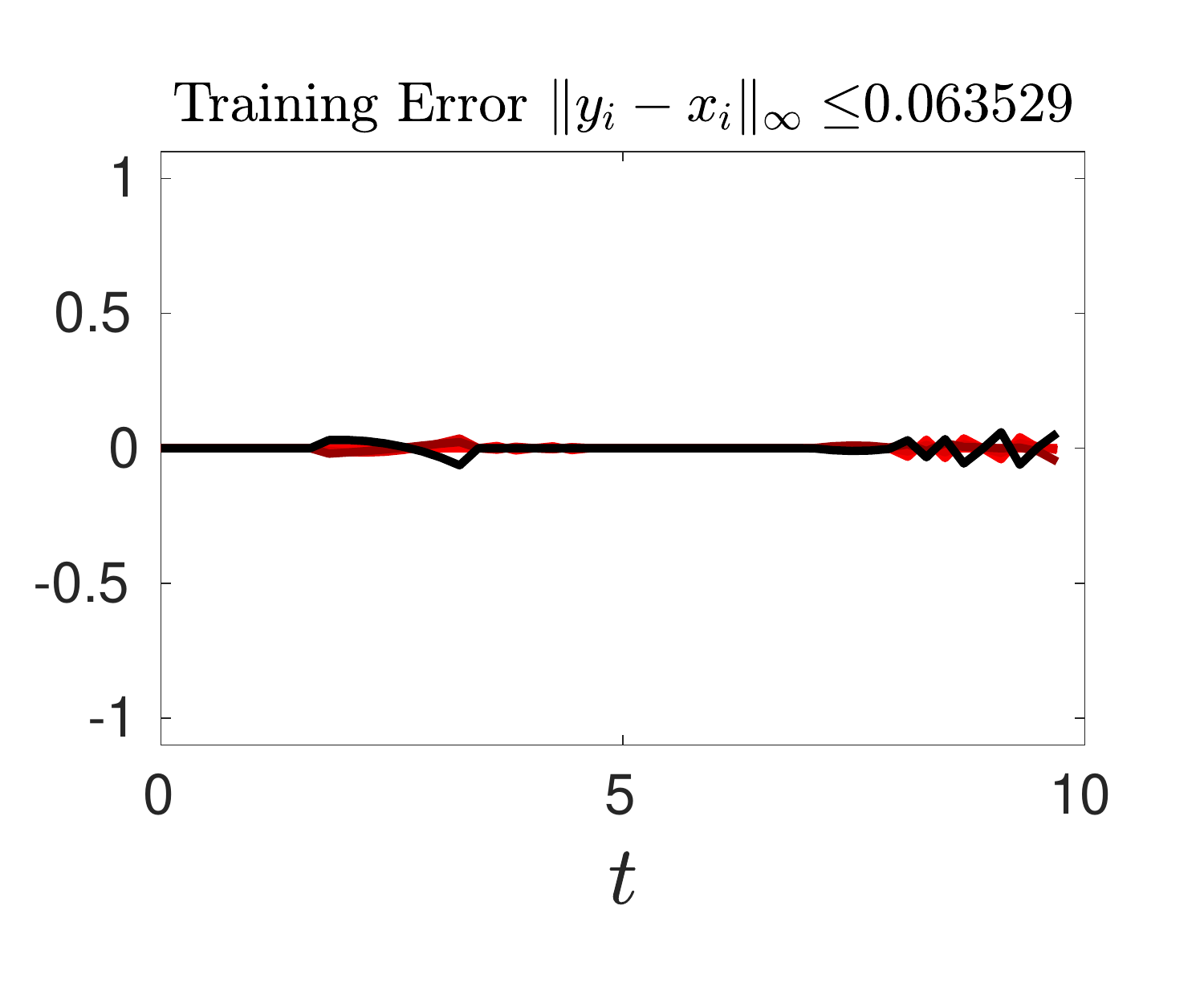}
    
    \caption{ {\footnotesize (Top-Left) Function $M(x)=\sin(0.25xt)$ (Top-Right) 5 equispaced samples in $(0,1)$  of the function $M$. The blue dotted line indicates the time $0$ (Center-Left) Generalization (Center-Right) Interpolation, trajectories in the training data. The blue dotted line indicates the time $0$ (Bottom-Left) Generalization error (Bottom-Right) Training error plot. 50 layers have been considered in the Euler discretization of \eqref{extended}. } }
    \label{fig:sim}
\end{figure}

\section{Conclusion}\label{s:conslusion}
{\color{black}
In contrast to first-order NODEs, Momentum ResNets and Memory NODEs can interpolate exactly any sample of points.
Precisely, given a sample $\{(x_i,y_i)\}_{i=1}^N \subset \R^d\times \R^d$ with $x_i\neq x_j$ for $i\neq j$ and $d, N\in\N$, Momentum ResNets and Memory NODEs can simultaneously control any point $x_i$ to its corresponding target $y_i$ in any given time interval (see our Theorem \ref{sim:control} and Proposition \ref{memorysimcontr}).
For first-order NODEs  as \eqref{NODE}, the analogous result (Theorem 2 in \cite{ruizbalet2021neural}) can be achieved only under the hypothesis $y_i\neq y_j$ for $i\neq j$. 
As explained in the Introduction, the hypothesis of not coinciding targets is necessary for first-order NODEs by uniqueness of the solution.
Instead, for the models investigated in this paper, two trajectories can have the same state variable, provided that they have different velocities or memory states. 

In the simulations (see Section \ref{ss:simus}), we have implemented a binary classification, where, at the final step, each poit is assigned to one of the two class according to a projection on a hyperplan.
Therefore, the decision boundary is a hyperplane after the transformation carried out by the flow of the NODE. Since the input-output map is a homeomorphism, the preimage of the hyperplane, and as a consequence also the decision boundary, is topologically equivalent to a hyperplane. This phenomenon is observed in the simulations. 
On the other hand, when considering a Momentum ResNet, the phase-space dimension is doubled in positions and velocities states, but we only focus on the position variable to determine the class of the point. Because of this property, we observed numerically that Momentum ResNets could have better approximation properties than first-order NODEs.

In Figures \ref{fig:Ng1} and \ref{fig:Mg1}, the task is to approximate the characteristic function of a circle through binary classification using first order NODEs and Momentum ResNets respectively. In Figure \ref{fig:Ng1}, one may observe that the decision boundary is not a closed curve whereas in Figure \ref{fig:Mg1} we can observe that the approximation performed by Momentum ResNets reflects the topology of the circle. 

In fact, when using Momentum ResNets, we assign null velocity to all initial data. The preimage of the final hyperplane is topologially equivalent to a hyperplane in $\R^{2d}$  because we are only observing the intput-output for data in $\mathbb{R}^d\times\{0\}$ (see Figure \ref{fig:decision}).
Hence, the performance of Momentum ResNets not only has presented a smaller overall error, but has been qualitatively preferable to first-order NODEs.

\begin{figure}
    \centering
    \includegraphics[scale=0.6]{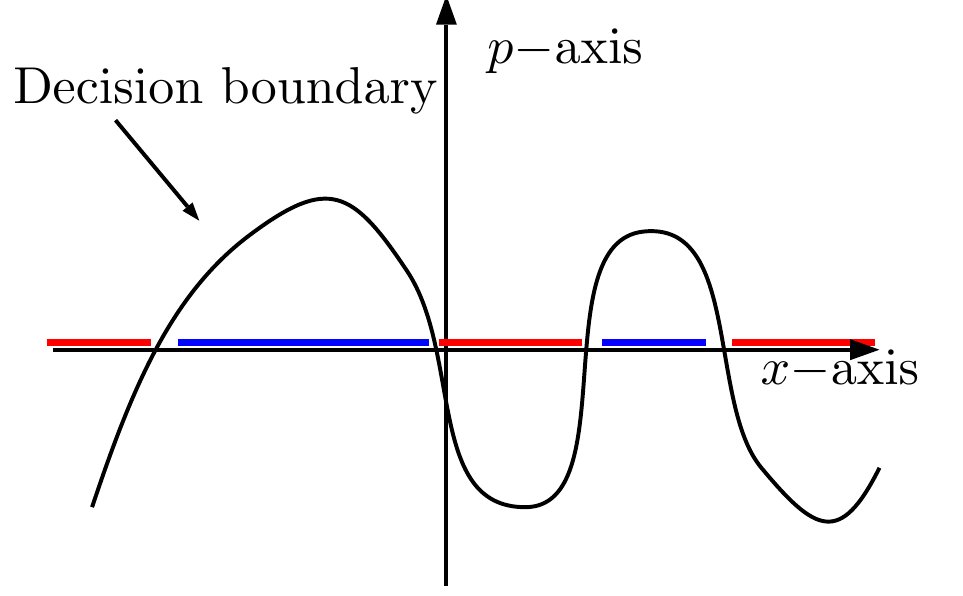}
    \caption{{\footnotesize Example of how, with the enlarged phase-space of Momentum ResNet, one can generate different topologies of the decision boundary. When restricting to initial data of the form $\R\times \{0\}$, the decision boundary can have different topologies. The colors red and blue represent the different classes in the binary classification.}}
    \label{fig:decision}
\end{figure}

On the other hand, in our Corollary \ref{thm:ua} for universal approximation with Momentum ResNets, we have a minimal controllability time which was not present in first-order NODEs. However, we don't know if this is a structural problem or if it is possible to prove this controllability result for any time interval with a different proof.

For Memory NODEs, we easily derived 
simultaneous controllability and universal approximation from the results of \cite{ruizbalet2021neural}, see  Proposition \ref{memorysimcontr}.
More importantly, we have shown the ability of Memory NODEs to approximate maps of the following type
$$C\left(\Omega;\bigg[C([0,T];\mathbb{R}^d)\cap BV([0,T];\mathbb{R}^d)\bigg]\right)$$
where $\Omega\subset\mathbb{R}^d$, see Theorem \ref{trackRNN} and Corollary \ref{cor:uta}. 

The theoretical result is also supported by the simulations (see Figure \ref{fig:sim}): with just five samples, we were able to have relatively good results in both simultaneous tracking and universal tracking approximation.

Memory issues are a central point in the assessment of a model, thus it is worth spending some words on this topic. 
For Momentum ResNets, the authors of \cite{sander2021momentum} were able to use a backpropagation algorithm that requires less memory than with first-order NODEs by exploiting the particular structure of the dynamics. 
Hence, Momentum ResNets appear to be not only more accurate but also more memory-efficient than first-order NODEs.

On the other hand, for the results concerning the Memory NODEs in \eqref{extended}, we need to require the dimension of the memory variable to be at least two times the dimension of the state variable.
However, we point out that the requirement on the memory size depends on the methods of our proof and we do not know if the same results can be achieved with less requirements.

We stress again that both models can be formulated in an integral form, namely 
$$
\dot{x}=-x+\int_0^tw(s)\sigma\left( \langle a(s),x(s)\rangle +b(s)\right)ds
$$
for the Momentum ResNet, and
$$
\dot{x}=W\vec{\sigma}\left( Ax  + C \int_0^t u\sigma\left(\langle d,x\rangle +f\right)ds +b_1\right)+b_2\\
$$
for the Memory NODE. 

A natural extension of this work could concern the study of better approximation properties, for instance, universal approximation in the $H^1$ norm. Furthermore, one could try to use Memory NODEs to obtain better results than the ones in \cite{ruizbalet2021neural} for transport equations, for example, trying an $L^1$-approximate controllability instead of a Wasserstein one.

}

\begin{appendix}

\section{Proofs: Momentum ResNets}\label{s:proofmr}

In this Subsection, we prove Theorem \ref{sim:control}. The proof is divided into three steps.
\begin{enumerate}[leftmargin=5mm]
 \item The first one is a preliminary step.
 \item The second concerns the control of the first component of each point.
 \item The third is related to the control of the remaining components for every point.
\end{enumerate}


{\color{black}

\begin{proof}

\begin{enumerate}[leftmargin=5mm]
 \item \textbf{Preparation of the data set}. Without loss of generality, one can ensure the following:
\begin{enumerate}
 \item There exists $C>0$ such that
 \begin{equation}\label{hyp1}
 \begin{split}
   \min_{i\in\{1,...,N\}}x_i^{(1)}>\max_{j\in\{1,...,N\}} y_j^{(1)}+C\\
 p_i^{(1)}\geq 0 \quad \forall i\in\{1,...,N\}
 \end{split}
 \end{equation}

 \item It holds
 \begin{equation}\label{hyp2}
   x_i^{(1)}(0)\neq x_j^{(1)}(0)\quad \forall i\neq j
 \end{equation}
\end{enumerate}

In order to do so, one can choose a hyperplane of the form $x^{(2)}+b=0$ such that all the data points belong to one side of such hyperplane, with $b>>1$ such that $x_i^{(2)}+b>0$ for all $i$. Then, by choosing $\vec{w}=(0,w^{(2)},0,\dots,0)$ with $w^{(2)}$ big enough, we can guarantee that after an arbitrarily small time all points fulfil \eqref{hyp1}. Indeed, by the variations of constants formula, one obtains
{\small
\begin{equation}\label{intw} 
\begin{split}
&\begin{pmatrix}x^{(1)}(t)\\p^{(1)}(t)\end{pmatrix}=\begin{pmatrix}x^{(1)}(0)+(1-e^{-t})p^{(1)}(0)\\ e^{-t}p^{(1)}(0)\end{pmatrix}+\\
&\int_0^t w^{(2)}(x^{(2)}(s)+b)\begin{pmatrix}1-e^{s-t}\\ e^{s-t} \end{pmatrix}ds.
\end{split}
\end{equation}
} 
For $b$ big enough, all terms in the above integral are positive  between in $(0,t)$. Thus, the flow increases the first component of every point.

Since the field is globally Lipschitz, there is uniqueness and global existence of the solution. This means that, after the procedure above, one has
$$(x_i,p_i)\neq (x_j,p_j)\quad \forall i,j\in\{1,...,N\}.$$
Take two points $i,j$, either \eqref{hyp2} is fulfilled, $x_i^{(1)}\neq x_j^{(1)}$, or there exists a component, $k$, in which the two points differ $x_i^{(k)}\neq x_j^{(k)}$. Take $b=-1/2(x_i^{(k)}+ x_j^{(k)})$ , $a=\delta_{k,k'}$ and $w=(w^{(1)},0,...,0)$. Then, by solving the dynamics for a time that can be choosen arbitarly small, one obtains that $x_i^{(1)}\neq x_j^{(1)}$. This process can be done for every pair of points that share the first component. Note that, we also have that $p_l^{(1)}\geq 0$ for all $l$.

Since this step can be done in arbitrary small time, we will assume that without loss of generality the initial data fulfills \eqref{hyp1} and \eqref{hyp2}.

 \item \textbf{Control the first component.} Since \eqref{hyp1} and \eqref{hyp2} are fulfilled at $t=0$, there exists a positive time $T_1$ such that in the interval $I_1=(0,T_1)$, \eqref{hyp1} and \eqref{hyp2} are fulfilled. Our goal now is to control the first component of each point. We will proceed iteratively, we split $I_1$ in $N$ intervals $I_1=\cup I_{1,i}=\cup (t_{i-1},t_{i})$ with $t_0=0$ and $t_i<\min\{T/2,T_1\}$ for $i=1,...,N$, to be determined later on. In the interval $I_{1,i}$ we will control the position and velocity of the point $i$ so that at the final time will fulfil that $x_i^{(1)}(T)=y_i$ (see Figure \ref{controlfree} for a qualitative representation). 
 However, to be able to achieve exact controllability, 
 we guarantee the property 
 \begin{equation}\label{condsum}
  x_i^{(1)}(t_i)+p_i^{(1)}(t_i)\neq x_j^{(1)}+p_j^{(1)}(t_j).
 \end{equation}
 Thanks to \eqref{condsum}, there does not exist an open interval $I\subset (T/2,T)$ such that
 $$x_i^{(1)}=x_j^{(1)}\quad \forall t\in I.$$
 Therefore, there exists an interval $I\subset (T/2,T)$ such that
 \begin{equation}\label{notequality}
     x_i^{(1)}(t)\neq x_j^{(1)}(t)\quad i\neq j\quad t\in I.
 \end{equation}
 Since \eqref{notequality} is fulfilled, the points can be ordered with respect to their first component and there is a hyperplane of the form $x^{(1)}=c$ separating any two points. The interval $I$ will be used for controlling the remaining components.

 We now ensure controllability of the first component and \eqref{condsum}.
 We will proceed iteratively; for every $i=1,...,N$ we do the following:
 \begin{enumerate}[leftmargin=2mm]
 \item Isolating the target point. Choose $a=(0,1,0,\dots,0)$ and $b$ such that $x^{(1)}_i(t)<b<x^{(1)}_{j}(t)$ for any $j$ such that $x^{(1)}_i(t)<x^{(1)}_{j}(t)$. As we proceeded in \eqref{intw}, one can generate a flow that pushes all the points such that $x^{(1)}_i(t)<x^{(1)}_{j}(t)$ towards $-\infty$ in the $x^{(2)}$ component. This process can be done arbitrarily fast and for any constant $\mathcal{C}$ one can find a control to ensure that
 $$p^{(2)}_j<-\mathcal{C} \text{ } \forall j\text{ s.t. } x^{(1)}_i(t)<x^{(1)}_{j}(t)\quad t \in I_1$$  
 The exact same procedure can be done for all points such that
 $$x^{(1)}_i(t)>x^{(1)}_{j}(t)\quad t \in I_1.$$
 Let us denote by  $t_{i,1}$ the time in which the following is fulfilled
 $$ x_j^{(2)}<x_i^{(2)}\text{ if }j\neq i$$
 $$ p_j^{(2)}<p_i^{(2)}\text{ if }j\neq i$$
 which can be always guaranteed by choosing the appropriate magnitude of $w$. Furthermore, let $t_{i,2}=t_{i}-t_{i,1}$. 
 \item Set the velocity on the first component of the target point so that at the final time reaches its target.
 \begin{equation*}
 \begin{split}
  y_i^{(1)}&=x^{(1)}_i(T)=x_i^{(1)}(t_{i,1})\\
  &+(1-e^{-t_{i,2}}p^{(1)}_i (t_{i,1})\\
  &+(1-e^{-(T-t_i)})e^{-t_{i,2}})p(t_{i,1})\\
  &q(t_{i,2}-1+e^{-t_{i,2}}\\
  &+(1-e^{-(T-t_i)})(1-e^{-t_{i,2}} ))
 \end{split}
 \end{equation*}
 where $q=w^{(1)}_i(x_{i}^{(2)}(t)+b(t))$ and $b(t)$ is chosen in a way that $q$ is time independent, i.e. $b(t)=b^*-x_{i}^{(2)}(t)$ and so that $x_{j}^{(2)}(t)+b(t)<0$ if $j\neq i$, $x_{i}^{(2)}(t)+b(t)>0$.
 
 Furthermore, $q$ depends on $t_i$, $q=q_{t_{i}}$, and it is equal to:
 \begin{equation*}\begin{split}
   q_{t_i}=\bigg( & t_{i,2}-1+e^{-t_{i,2}} \\ &+(1-e^{-(T-t_i)})(1-e^{-t_{i,2}})\bigg)^{-1}\\
     &\bigg( y_i^{(1)}-x^{(1)}_i(T)-x_i^{(1)}(t_{i,1})\\
     &-(1-e^{-t_{i,2}})p^{(1)}_i(t_{i,1})-\\
     &(1-e^{-(T-t_i)})e^{-t_{i,2}}p(t_{i,1})\bigg)   
   \end{split}
\end{equation*}

 \item Ensuring \eqref{condsum}. Note that 
 $$x^{(1)}(t)+p^{(1)}(t)=x^{(1)}(0)+p^{(1)}(0)+q_{t_i}(t_{i}-t_{i,i})$$
 and that $q_{t_i} (t_i-t_{i-1})$ is not constant. Therefore, we can always choose $t_i$ so that \eqref{condsum} is satisfied.

 \end{enumerate}

 \item \textbf{Control the remaining components} in the interval $I$.
 
 Thanks to \eqref{notequality}, we can relabel the points according to the first component, i.e.
 $$i<j \implies x_i^{(1)}(t)< x_j^{(1)}(t)\quad t\in I .$$
 One can proceed iteratively. From $i=1,...,N$ one chooses a hyperplane of the form $x^{(1)}-x^{(1)}_i+r$ with $r>0$ so that
 $$ x^{(1)}_j-x^{(1)}_i+r<0\quad \forall j<i$$
 and then, choose  solve the dynamics with $w=(0,w^{(2)},w^{(3)},...,w^{(d)})$  in a small time interval $I_{2,i}\subset I$ so that the point $i$ at the final time $T$ reaches the target. See Figure \ref{illustration} for an illustation of the procedure
 
\end{enumerate}

\end{proof}

\begin{figure*}
\centering
\begin{subfigure}{.45\textwidth}
  \centering
 \includegraphics[scale=0.23]{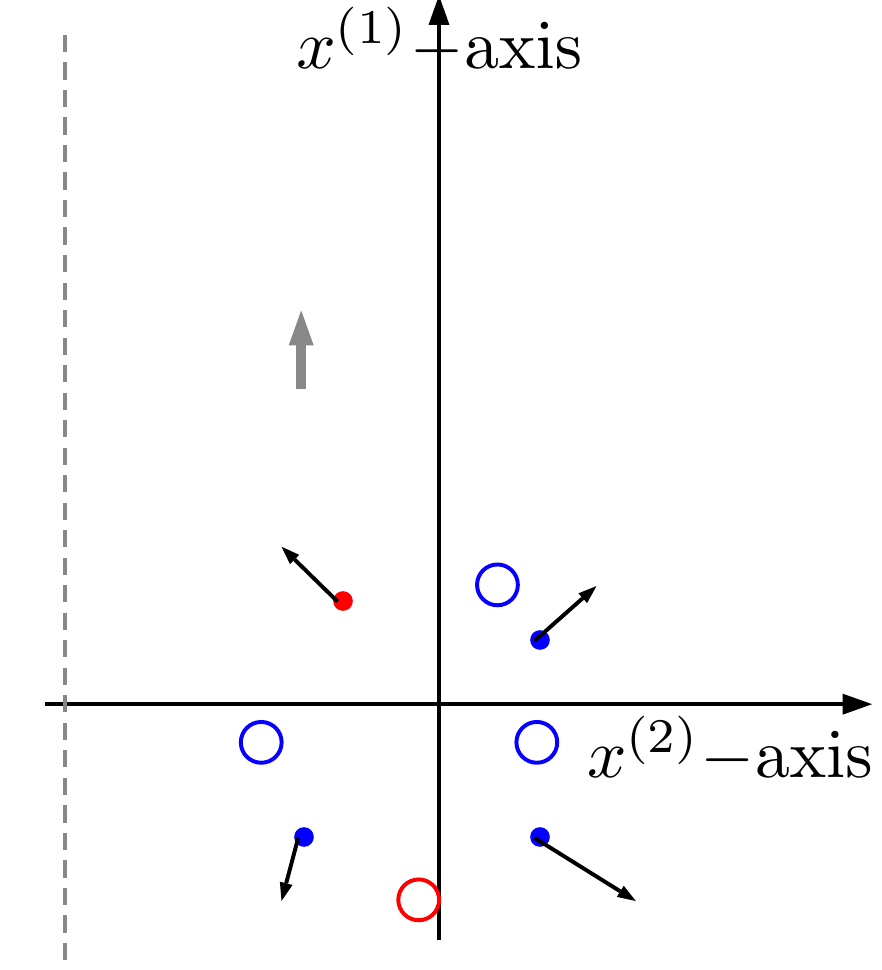}
  \includegraphics[scale=0.23]{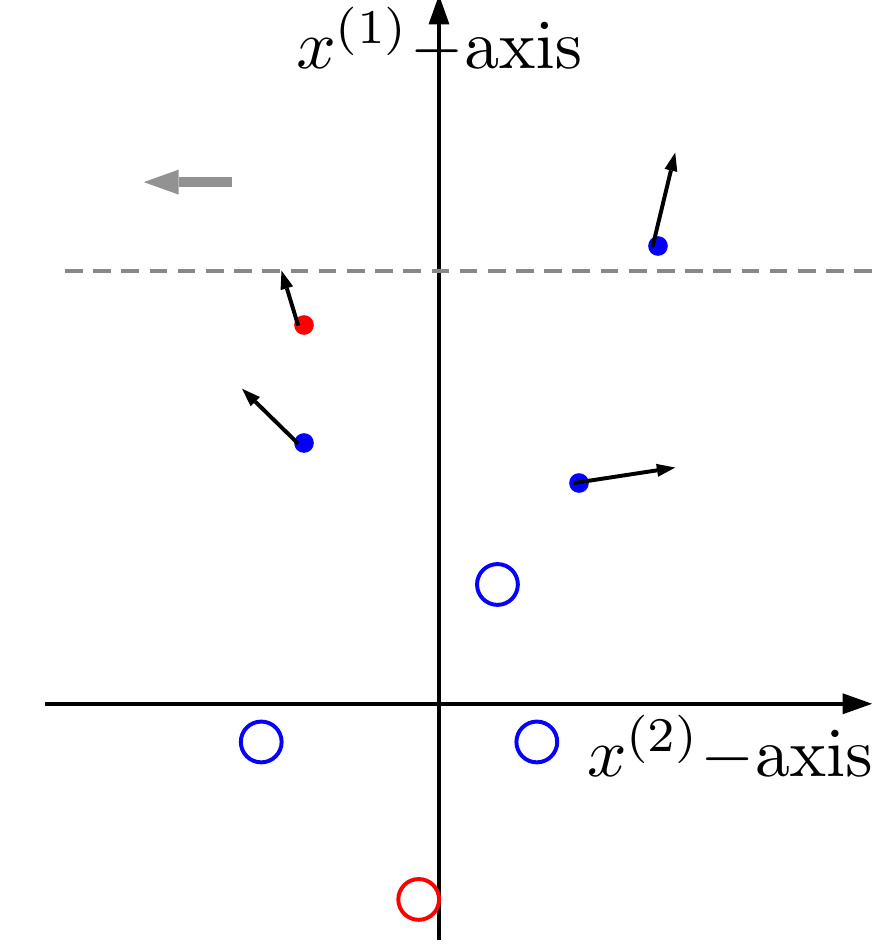}
  
   \includegraphics[scale=0.23]{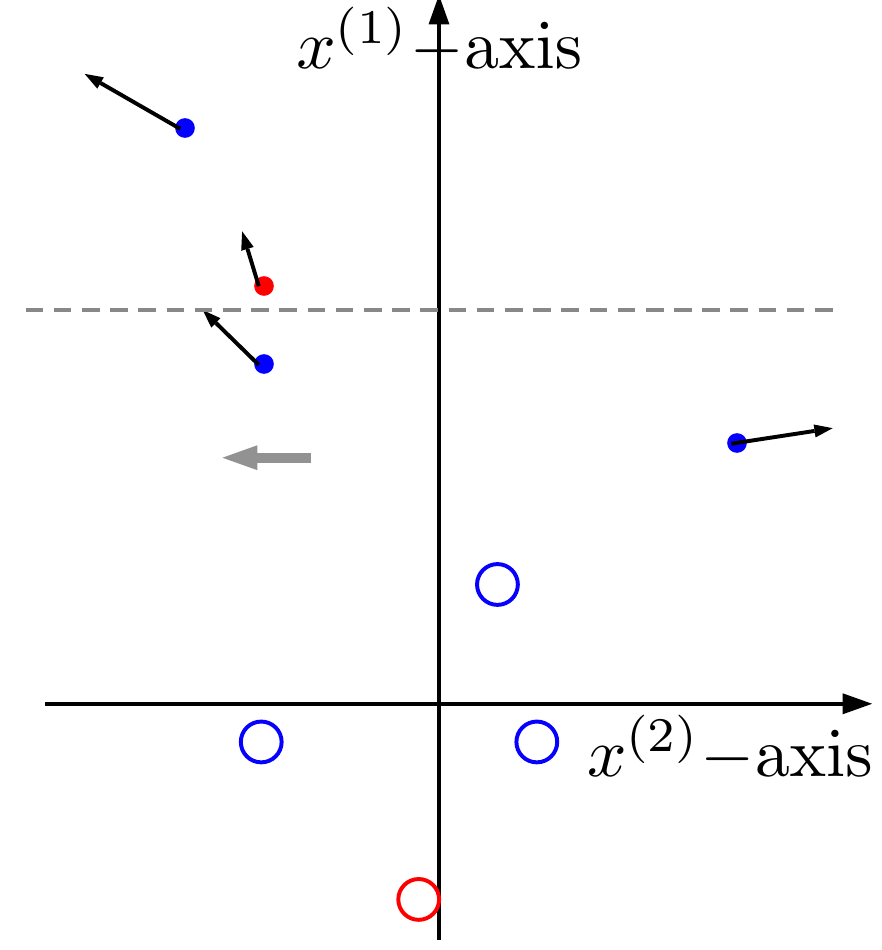}
 \includegraphics[scale=0.23]{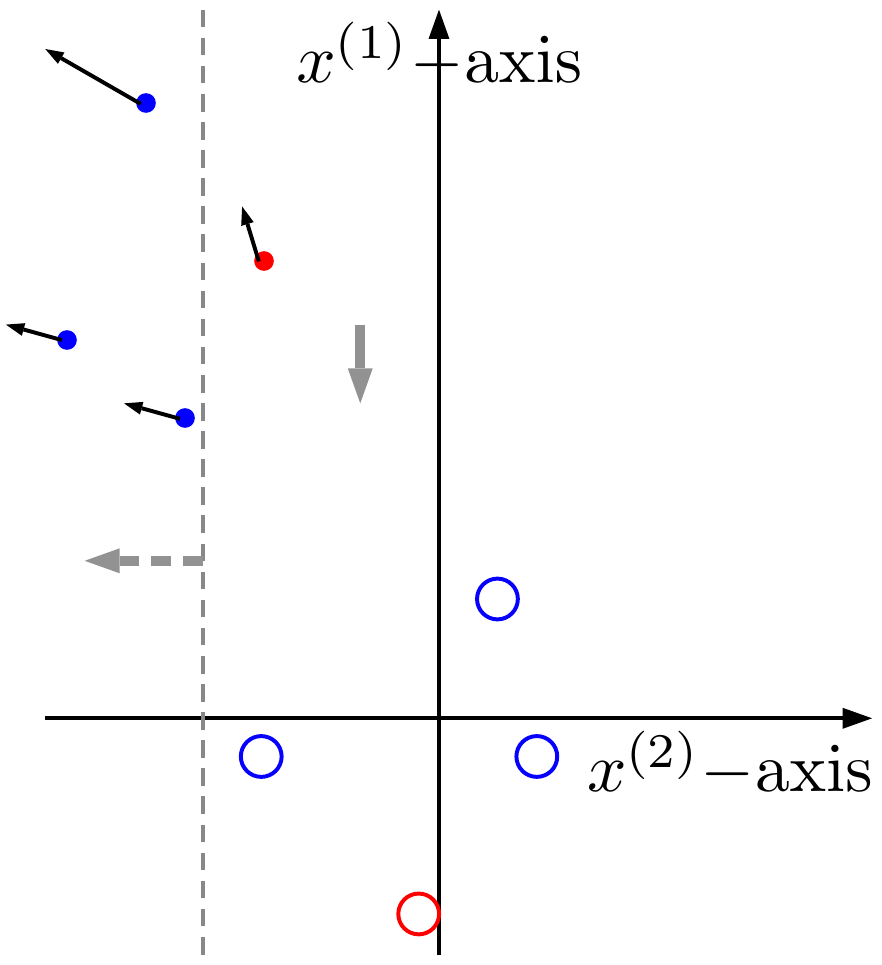}
  \caption{}
\end{subfigure}
\begin{subfigure}{.45\textwidth}
  \centering
 \includegraphics[scale=0.23]{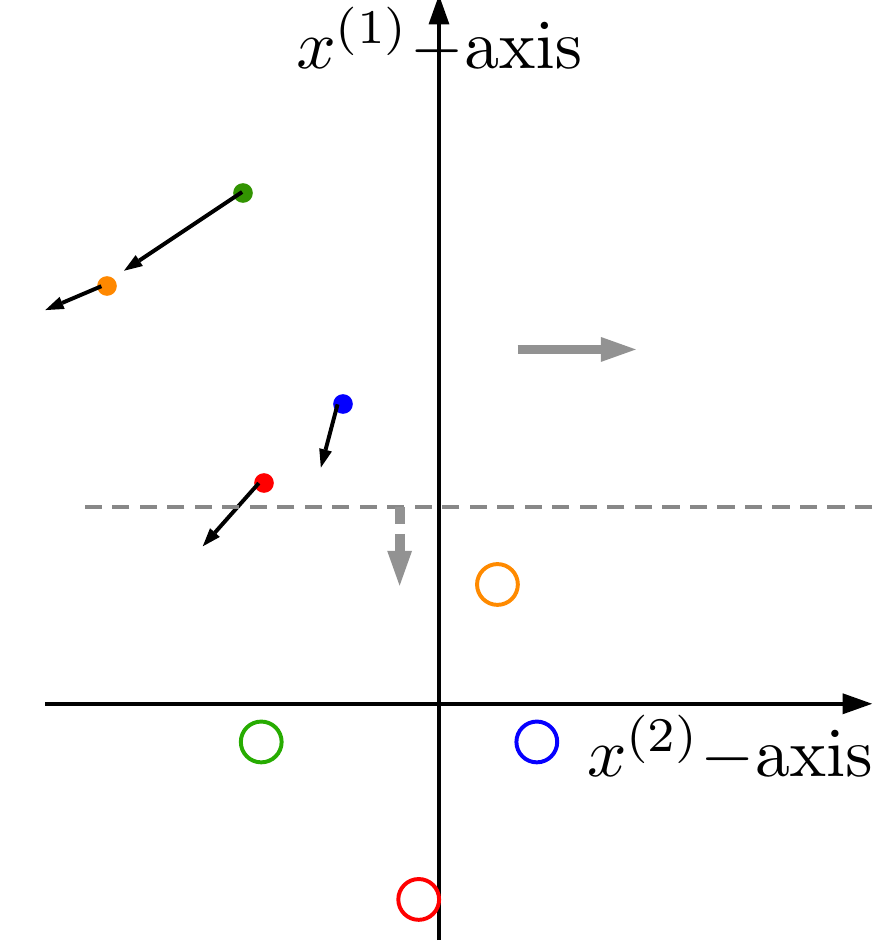}
  \includegraphics[scale=0.23]{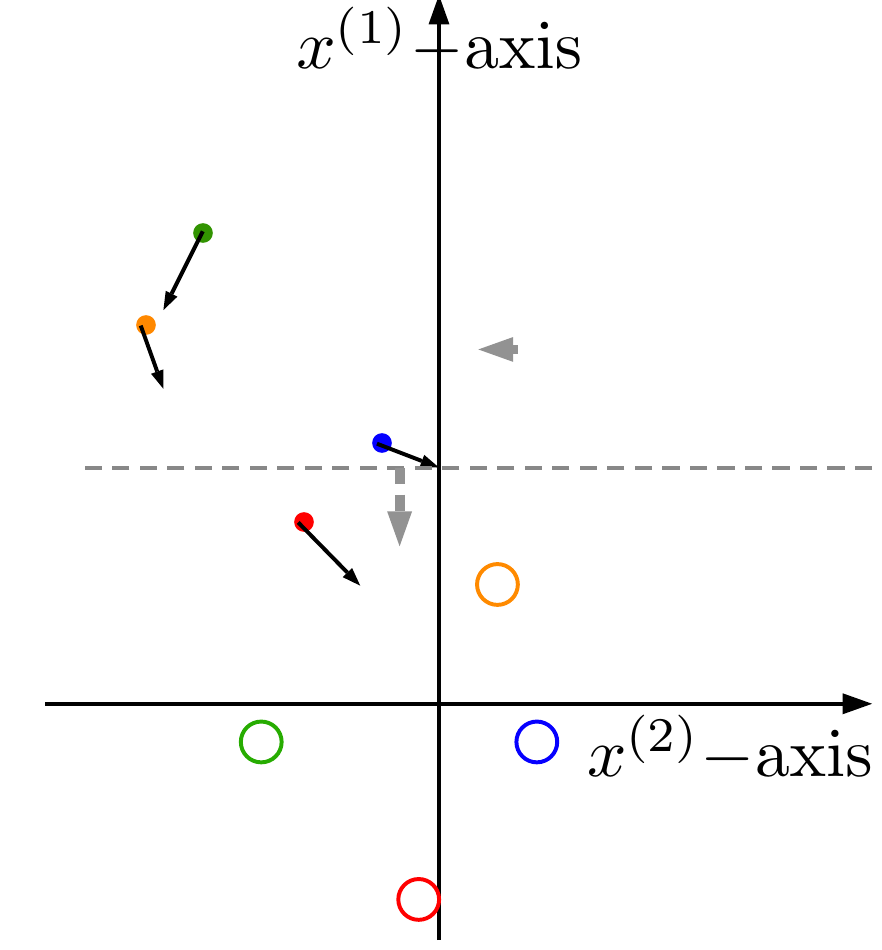}
  
 \includegraphics[scale=0.23]{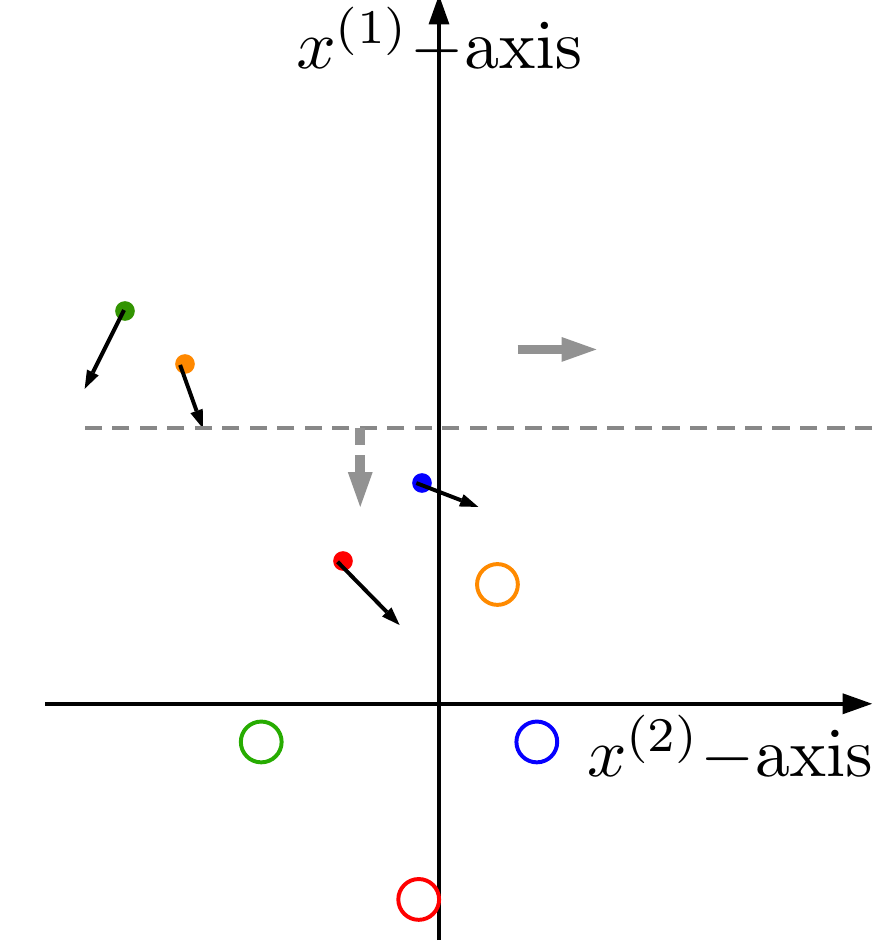}
 \includegraphics[scale=0.23]{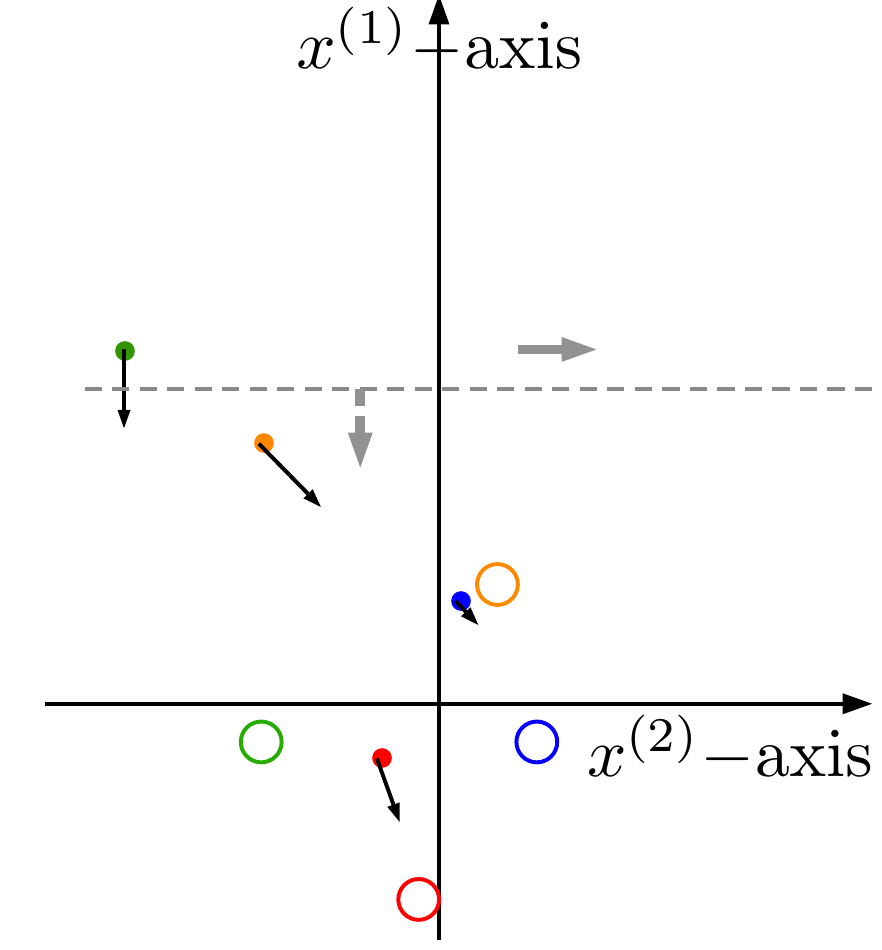}
  \caption{}
\end{subfigure}
\caption{ {\footnotesize  The black arrows represent the direction of the velocity of each point, the circles the target locations of the points. The dashed line is the chosen hyperplane and the gray arrow the direction of the field. The subfigure (a) represents the control of the first coordinate of the red point. (Top left) Preparation of the data set. 
(Top right) We choose a hyperplane slightly above the point we want to control its first component and we push all points above (with respect to the $x^{(1)}$ coordinate) the target point to $-\infty$ in the second component of the system. This process should be done fast enough, so that no point below the hyperplane will cross it. (Bottom-left) The same is applied for all points below the target point. (Bottom-right) The grey dotted arrow attached to the hyperplane indicates that the hyperplane is moving to keep the distance between the target point constant. The target point can be separated from all the others by a hyperplane. We can assign the appropriate velocity in the first component of the system so that it coincides with its target at the final time. This process will be repeated for every point. Subfigure (b). The grey dotted arrow attached to the hyperplane indicates that the hyperplane is moving to keep the distance between the point that has to be controlled and the hyperplane constant. The figure represents how to control the components $x^{(i)}$ for $i=2,...,d$. (Top left) Choice of the hyperplane. (Top right) We choose a hyperplane slightly below the point we want to control and we act on the velocity of every point so that the first one has the desired velocity to reach the target. (Top right) After having acted on the first point we chose a hyperplane above the point that has been controlled so that the next action does not affect it and we act on the remaining points so that the second point has the desired velocity to reach its target. (Bottom) The procedure follows iteratively until all points will reach their location at the final time.}}\label{illustration}
\end{figure*}
}
\section{Proofs: Neural ODE model}\label{appsimcontr}\label{a:ApC}

In this appendix we will prove the simultaneous control of 
\begin{equation}\label{reduced system}
    \begin{cases}
    x'=w(t)\sigma(a(t)x+c(t)p+b(t)\\
    p'=u(t)\sigma(d(t)x+f(t)\\
    \end{cases}
\end{equation}
where $w,u,a,c,b,f\in L^\infty((0,T);\mathbb{R})$ are control functions. The proof follows from Theorem 2 in \cite{ruizbalet2021neural} in a straightforward manner. For the sake of completeness, we will provide here the needed proof in our situation.
\begin{lemma}\label{A:simcontr}
Let $\sigma$ be the ReLU and consider $\{(x_i,p_i)\}_{i=1}^N\subset \mathbb{R}^2$ be distinct initial data and \newline $\{(y_i,\varphi_i)\}_{i=1}^N\subset \mathbb{R}^2$ be distinct target data. Then, for every time horizon $T>0$, there exist controls  $w,u,a,c,b,f\in L^\infty((0,T);\mathbb{R})$ such that the solution at time $T$ satisfies
$$\phi_T(x_i;\omega)=(y_i,p_i) \quad i\in\{1,...,N\}.$$
\end{lemma}

\begin{proof}
The proof is divided in different steps,
 \begin{enumerate}[leftmargin=5mm]
     \item First we will find a set of equivalent initial data points and equivalent targets that satisfy convenient properties.
     \item Secondly we will control the state component of the system.
     \item And finally we will control the memory component of the system.
 \end{enumerate}
 
 Note that a trajectory of \eqref{reduced system} in $(0,T)$ can be seen as a trajectory in the time interval $(0,1)$ up to a rescaling on the time that will be absorbed by the controls $w$ and $u$. 
 
 \begin{enumerate}[leftmargin=5mm]
     \item We will find control functions that allow us to assume without loss of generality that:
     \begin{equation}\label{HA1}x_i\neq x_j\qquad \text{if }i\neq j \end{equation}
     and
     $$y_i\neq y_j\qquad \text{if }i\neq j $$
     Assume that \eqref{HA1} is not satisfied by only two points, let us say $x_1=x_2$. Then, since the points $\{(x_i,p_i)\}_{i=1}^N$ are distinct, we can fix the hyperplane 
     $$p+b=0$$
     with $b=1/2(p_1+p_2)$. This hyperplane separates $(x_1,p_1)$ and $(x_2,p_2)$. Since the activation function is zero in one half space, then by just setting $w=1$ and $a=0$ and solving the ODE \eqref{reduced system} we will generate a flow as in Figure \ref{previousflows}. We solve \eqref{reduced system} for $T$ small enough so that the solution at time $T$ satisfies:
     $$x_i(T)\neq x_j(T)\qquad \text{ if } i\neq j$$
     Note that this argument can be applied iteratively for any number of coincidences on the state.
     
     On the other hand, note that the backward trajectories of \eqref{reduced system}
     are also forward trajectories of \eqref{reduced system} with the appropriate changes in the control functions. This allows us to apply the same argument for the target points whenever there is a coincidence on the target states.
     
      \item For controlling the first component of each data point we will apply recursively three flows.
      
      For each point we apply the following procedure:
       \begin{enumerate}[leftmargin=2mm]
           \item Choose the hyperplane $x-x_i+\epsilon=0$ (with $d=1$ $f=-x_i+\epsilon$) for $\epsilon>0$ small enough so that 
           $$\min_{j} |x_j-x_i+\epsilon|=\epsilon $$
           then we chose $u=1$ and we solve \eqref{reduced system} for $T$ large enough so that
           $$p_i<p_j(T) \quad \text{if }x_j>x_i $$
           \item we repeat the same procedure but by choosing the hyperplane $x-x_i-\epsilon=0$ ($d=-1$ $f=x_i+\epsilon$) for $\epsilon>0$ small enough so that 
           $$\min_{j} |x_j-x_i-\epsilon|=\epsilon $$
           then we chose again $u=1$ and we solve \eqref{reduced system} for $T$ large enough so that
           $$p_i<p_j(T) \quad \text{if }x_j<x_i $$
           
           \item In this way, one has that the point $(x_i,p_i)$ remained without moving and 
           $p_j(T)>p_i$ for all $j$. This means that there exist a hyperplane of the form $p-p_j-\epsilon=0$ (with $c=-1$ $b=p_i+\epsilon$) that separates the point $(x_i,p_i)$ from all the others. This allows us generate a flow that is able to exactly control $x_i$ to $y_i$
       \end{enumerate}
       Up to now, we have been able to control exactly the state of each point.

      \item To control the memory, we first relabel the points in a way that $y_i<y_j$ if $i<j$. For controlling the memory, it is enough to iteratively choose $d=1$ and $f=y_i-\epsilon$ for $\epsilon$ small enough and chose $u$ so that at time $T$ the memory of the point $i$ is controlled. Note that by the procedure above, once a point has been exactly controlled, the next stages do not perturb that point again.
     
 \end{enumerate}
\end{proof}


\subsection{Proof of Theorem \ref{memorysimcontr}}\label{ApC}

\begin{proof}[Proof of Theorem \ref{memorysimcontr}]
We will split the proof in several a cases. The main idea is to group a memory component with a state component and to apply the strategy in \cite{ruizbalet2021neural}.

\textbf{Case $\boldsymbol{d=d_p}$.}
In this case we can consider the pairs
$$(x^{(k)},p^{(k)})\quad \text{for }k=1,..,d $$
The strategy will rely on applying the strategy of \cite{ruizbalet2021neural} to each pair state-memory, see Appendix Lemma \ref{A:simcontr}.

\textbf{Case $\boldsymbol{d>d_p}$.}
As we have noted already, due to the particular strategy of \cite{ruizbalet2021neural}, even though the memory equation does not depend on itself, we can nonetheless have the necessary controls to perform the simultaneous control when we consider pairs of state and memory.

In this case, the dimension of the state is bigger than the dimension of the memory. This invites to consider
$$(x^{(k)},p^{(k)})\quad \text{for }k=1,..,d_p-1 $$
where the previous case immediately applies 
and consider the remaining components altogether
$$(x^{(d_p)},x^{(d_p+1)},...,x^{(d)},p^{(d_p)}) $$

\textbf{Case $\boldsymbol{d<d_p}$}
As mentioned in the previous cases, since the memory equation does not depend on itself, to apply the strategy of \cite{ruizbalet2021neural} we need to group memory variables with state variables. In this case, we have more memory variables than state variables. However we can still manage to reduce it to a simpler case by first controlling $d_p-d$ memory variables using one state component and then reduce to the $d=d_p$ case.

Split the time interval $(0,T/2)$ into $d_p-d$ time subintervals, $I_k$, in each time interval we will consider the pair
$$(x^{(1)},p^{(d+k)})\quad \text{for }k=1,..,d_p-d $$
and control the $p^{(d+k)}$ component.

Afterwards, we are left with the pairs
$$ (x^{(k)},p^{(k)})\quad\text{for }k=1,..,d$$
uncontrolled, but this is resolved with the first case.

\end{proof}

 \subsection{Proof of Theorem \ref{trackRNN}}\label{ApC2}

 \begin{proof}[Proof of Theorem \ref{trackRNN}]
 We will prove the theorem for $d_p=2d$ and this implies also the case $d_p>2d$ by simply ignoring the $d_p-2d$ extra variables. We will divide the proof in several steps
 \begin{enumerate}[leftmargin=5mm]
     \item \textbf{Approximation.} Consider a partition of $(0,T)$ by  $N_t$ intervals $I_k$, $(0,T)=\cup_{k=1}^{N_t} I_k$ so that every $y_i\in C^0((0,T);\mathbb{R}^d)\cap BV((0,T);\mathbb{R}^d)$ is approximated by a piecewise linear function. We consider the piecewise linear function to be all of them linear in every $I_k$ as Figure \ref{fig:piecwise} shows. Let $y_i^h$ be the approximation for $y_i$ and let $y_{i,k}^h$ be the function in the interval $I_k$
     $$\sup_{0\leq t\leq T} |y_i(t)-y_i^h|\leq \epsilon\qquad i\in\{1,...,N\} $$
     
     Furthermore, we can consider an approximation that fulfills
     \begin{equation}\label{hyptrack}
     \begin{array}{l}
      \partial_{x^{(1)}} y_{i,k}^h \neq \partial_{x^{(1)}} y_{j,k}^h \\ \text{if }j\neq i\quad k\in\{1,...,N_t\}.
     \end{array}
     \end{equation}
     Thanks to \eqref{hyptrack} for every $k$ there exist a subinterval $\tilde{I}_k\subset I_k$ such that
     \begin{equation}\label{proptrack}
     \begin{array}{l}
      \left(y_{i,k}^h(x) \right)^{(1)}\neq \left(y_{j,k}^h(x)\right)^{(1)}\\ \forall x\in \tilde{I}_k, \quad \text{if }j\neq i, \quad k\in\{1,...,N_t\}
     \end{array}
     \end{equation}

     \item \textbf{Simultaneous control.} Now we apply Theorem \ref{memorysimcontr} to approximately simultaneously control the state and the memory. We choose for every $i$ we choose $y_{i,1}^h(0)$ as targets for the state component. The targets in the memory component will be chosen accordingly so that the dynamics
     $$x_i'=\sigma(Cp_i)+b_2 $$
     is able to reproduce the required linear trajectories. We discuss the target selection in the next step.
     
     \item \textbf{Alternating strategy.} As we mentioned before, we will make an abuse of notation and consider $p=(p^{(1)},p^{(2)})$ with $p^{(1)},p^{(2)}\in \mathbb{R}^d$. The strategy is visualised in Figure \ref{fig:memdiag}; we will use one component of the memory to endow a linear movement while, in parallel, we will reconfigure the other component of the memory for the approximation in the next interval. This requires to understand two things, first, which targets for the memory we should select (Memory use) and second, how to control the component of the memory not used for controlling (Reconfiguration).
     \begin{enumerate}[leftmargin=2mm]
         \item \textit{Use of the Memory.} Let us consider the linear target functions in $I_k$:
         $$y_{i,k}^h=B_{i,k}+P_{i,k}t,\quad i\in\{1,...,N\},$$
         with $P_{i,k},B_{i,k}\in \mathbb{R}^d$. Consider $A=0,b_1=0,W=I_d$, the dynamics of the state equation is
     $$x'=\boldsymbol{\vec{\sigma}}(Cp)+b_2.$$
     We choose $\mathbb{R}^{d\times d_p}\ni C=(I_d|0)$ and we look for memory targets that are in the positive cone $\{p_{i,k}^{(1)}\}_{i=1}^N \subset\mathbb{R}_{+}^d$. The activation function $\sigma$ is assumed  to be the ReLU, and hence, we reduce the problem on finding $b_{2,k}\in \mathbb{R}^d$ such that
     $$p_{i,k}^{(1)}=P_{i,k}-b_{2,k}\in \mathbb{R}_{+}^d\qquad i\in \{1,...,N\},$$
     which is always possible. Then the dynamics trivially reads
          $$
          x'=P_{i,k},\quad      x(0)=B_{i,k}
          $$
          whose solutions are $P_{i,k}t+B_{i,k}$ for all $i\in\{1,...,N\}$.
         
         

         \item \textit{Reconfiguration.} Since, for every $k$ we have that there exists an interval $\tilde{I_{k}}$ for which \eqref{proptrack} holds, one can control the second component of the memory while the first component is used for following the trajectory. Thanks to \eqref{proptrack} and the continuity of the trajectories, we can consider a subinterval $\dbtilde{I}_{k}$ such that, up to a relabeling of the points we have
         $$\left( y_{i,k}^h \right)^{(1)}<\left(y_{j,k}^h\right)^{(1)}\quad t\in\dbtilde{I}_k,\quad \text{if } i<j.$$
         We set $d=(1,0,...,0)$ and  we split $\dbtilde{I}_k$ into $N$ subintervals, name them $\dbtilde{I}_{k,i}$ and we proceed to control the memory component that is not used to control the state sequentially. Set 
         $$f(t)=\left( y_{i,k}^h(t) \right)^{(1)}-\delta $$
         for $\delta>0$ small enough so that
         $$\left( y_{j,k}^h(t) \right)^{(1)}<f(t)\quad t\in \dbtilde{I}_{k,i}\quad \text{if }j<i $$
         We settle the memory targets for the next interval $I_{k+1}$ as in the previous step and then we choose
         $$ u=\beta \left( p_{i,k+1}-p_{i,k-1}\right)$$
         for $\beta\in \mathbb{R}$ so that when we solve the ODE
         $$p'=u\sigma(\langle d,x\rangle +f) \quad t\in \dbtilde{I}_{k,i} $$
         the solution at $t_{k,i}:=\sup \dbtilde{I}_{k,i}$ satisfies that $p_i(t_{k,i})=p_{i,k+1}$.
         Note that if $j<i$, the equation becomes
         $$p'=0$$
         so, once we have  controlled the memory for one point, in the subsequent steps this memory will be unaltered. Therefore, we are able to exactly control the memory component where we desire.

     \end{enumerate}

 \end{enumerate}
 \end{proof}

\end{appendix}
\phantomsection
\section*{Acknowledgments} 

We sincerely thank Carlos Esteve-Yagüe for his valuable conversations and comments.

This work has been funded by the European Research Council (ERC) under the
European Union’s Horizon 2020 research and innovation programme (grant agreement NO: 694126-DyCon). The work of the third author is also partially supported by the Air Force Office of Scientific Research (AFOSR) under Award NO: FA9550-18-1-0242, by the Grant MTM2017-92996-C2-1-R COSNET of MINECO (Spain), by the Alexander von Humboldt-Professorship program, the European Unions Horizon 2020 research and innovation programme under the Marie Sklodowska-Curie grant agreement No.765579-ConFlex, and the Transregio 154 Project “Mathematical Modelling, Simulation and Optimization Using the Example of Gas Networks” of the German DFG.

\bibliographystyle{abbrv}
\bibliography{biblio.bib}

\end{document}